\documentclass[a4paper,twoside,reqno, 12pt]{amsart}
\usepackage{amssymb}
\usepackage{mathrsfs}
\usepackage[pdftex,
colorlinks, 
bookmarksnumbered, 
bookmarksopen, 
pdfstartview=FitH, 
linkcolor=blue, 
citecolor=blue, 
urlcolor=magenta, 
filecolor=blue %
]{hyperref}
\usepackage{keyval}

\providecommand {\norm}[1] {\lVert#1\rVert}
\providecommand {\abs}[1] {\lvert#1\rvert}
\providecommand {\inprod}[1]{\langle #1 \rangle}
\providecommand {\set}[1]{\lbrace #1 \rbrace}

\newcommand {\bbN} {\ensuremath{\mathbb{N}}}
\newcommand {\bbR} {\ensuremath{\mathbb{R}}}
\newcommand {\bbZ} {\ensuremath{\mathbb{Z}}}
\newcommand {\bbT} {\ensuremath{\mathbb{T}}}
\newcommand {\bbC} {\ensuremath{\mathbb{C}}}

\newcommand {\bbNd} {\ensuremath{\mathbb{N}^d}}
\newcommand {\bbRd} {\ensuremath{\mathbb{R}^d}}
\newcommand {\bbTd} {\ensuremath{\mathbb{T}^d}}
\newcommand {\bbCd} {\ensuremath{\mathbb{C}^d}}

\newcommand {\bbZd}[1] [d] {\ensuremath{\bbZ^{#1}}}
\newcommand {\mA} {\ensuremath{\mathcal{A}}}
\newcommand {\mD} {\ensuremath{\mathcal{D}}}
\newcommand {\mE} {\ensuremath{\mathcal{E}}}
\newcommand {\mF} {\ensuremath{\mathcal{F}}}
\newcommand {\mJ} {\ensuremath{\mathcal{J}}}
{
\newcommand {\mX} {\ensuremath{\mathcal{X}}}

\newcommand {\mM} {\ensuremath{\mathcal{M}}}

\newcommand {\mB} {\ensuremath{\mathcal{B}}}

\newcommand {\mC} {\ensuremath{\mathcal{C}}}
\newcommand {\mS} {\ensuremath{\mathcal{S}}}
\newcommand {\mT} {\ensuremath{\mathcal{T}}}

\newcommand {\bopzd}  {\ensuremath{\mathcal B \bigl( \ell^2(\bbZd)\bigr)}}
\newcommand {\bop}  {\ensuremath{\mathcal B ( \ell^2)}}
\newcommand {\bopn}  {\ensuremath{\mathcal B ( \ell^2(\bbN))}}

\DeclareMathOperator{\Lip}{Lip}

\DeclareMathOperator{\supp}{supp}
\DeclareMathOperator{\spann}{span}
\DeclareMathOperator{\dd}{\mathrm{d}}

\DeclareMathOperator{\spec}{spec}

\newcommand \AT {approximation theory}

\newcommand \BL {bandlimited}
\newcommand \tfae {The following are equivalent:}

\newcommand \cont {continuous}
\newcommand \BS {Banach space}
\newcommand \BA {Banach algebra}
\newcommand \MA {matrix algebra}

\newcommand \odd {off-diagonal decay}

\newcommand \IC {inverse-closed}

\newcommand \HZ {H\"{o}lder-Zygmund}
\newcommand \LP {Littlewood-Paley}
\newcommand \DPU {dyadic partition of unity}
\newcommand {\cexp} [1] [{x t}] {\ensuremath{e^{2 \pi i #1}}}

 \newcommand {\bigo}{{\mathcal{O}}}

\newtheorem{prop}{Proposition} [section]
\newtheorem{cor}[prop]{Corollary}
\newtheorem{thm}[prop]{Theorem}

\newtheorem{lem}[prop]{Lemma}

\theoremstyle{definition}
\newtheorem{defn}[prop]{Definition} 

\theoremstyle{remark}
\newtheorem*{rem}{Remark}
\newtheorem*{rems}{Remarks}
\newtheorem{ex}[prop]{Example}

\numberwithin{equation}{section}

\newcommand {\domd} {\mD(\delta)} 
\newcommand {\one} {\mathbf{1}}

\newcommand{\inv}{^{-1}}

\begin{document}
\title [Noncommutative approximation] {Noncommutative Approximation:\\
  Inverse-Closed Subalgebras and Off-Diagonal Decay of Matrices} \date{\today} \author{Karlheinz Gr\"ochenig}
\author{ Andreas Klotz}
\address{Faculty of Mathematics\\University of Vienna \\
  Nordbergstrasse 15\\A-1090 Vienna, AUSTRIA} \email{karlheinz.groechenig@univie.ac.at}
\email{andreas.klotz@univie.ac.at} \thanks{K.~G.~was supported by the Marie-Curie Excellence Grant MEXT-CT
  2004-517154, A.~K.~ was supported by the grant MA44 MOHAWI of the Vienna Science and Technology Fund (WWTF), 
and partially by the National Research Network S106 SISE of the
Austrian Science Foundation (FWF)}

\begin{abstract}
We investigate two systematic constructions of inverse-closed
subalgebras of a given Banach algebra or operator algebra $\mA $, both of
which are inspired by classical principles of approximation theory. 
The first construction requires a closed derivation or a commutative
automorphism group on $\mA $ and yields a family of smooth 
inverse-closed subalgebras of $\mA $ that resemble the usual
H\"older-Zygmund spaces. The second construction starts
with a graded sequence of subspaces of \mA\ and yields a class of
inverse-closed subalgebras that resemble the classical approximation
spaces. We prove a theorem of Jackson-Bernstein type to show that in
certain cases both constructions are equivalent. 

These  results about abstract Banach algebras are applied to algebras
of infinite matrices with off-diagonal decay. In particular, we obtain
new and unexpected conditions of off-diagonal decay that are preserved under
matrix inversion. 
 \end{abstract}
\maketitle

\section{Introduction}

 A remarkable class of results in numerical analysis asserts that the
 off-diagonal decay of an infinite matrix is inherited by its inverse
matrix. The prototype is Jaffard's theorem~\cite{Jaffard90}: If the matrix $A$
with entries $A(k,l)$, $k,l\in \bbZ$, is invertible on $\ell ^2(\bbZ )$ and if,
 for $r>1$, 
 \begin{equation}
   \label{eq:1}
   |A(k,l)| \leq C (1+|k-l|)^{-r} \qquad \text{ for all } k,l \in \bbZ
   \, ,
 \end{equation}
then also 
$$
|(A\inv )(k,l)| \leq C (1+|k-l|)^{-r} \qquad \text{ for all } k,l \in
\bbZ \, .
$$
This result has found many variations and inspired a long line of
research. Off-diagonal decay has been modeled (a) with more general
weight functions ~\cite{Bas90,Baskakov97, Demko84,GL04a}, (b) with weighted
versions of Schur's test~\cite{GL04a},  (c) with convolution-dominated
matrices~\cite{Bas90,GKW89,Gr08,Kur90,Sjo95}, or (d) with  mixtures of
such conditions~\cite{Sun07a}.

Which forms of \odd\ are inherited by the inverse of a 
matrix? The answers  so far mix art and  hard mathematical
work. The art is  to guess  a  suitable decay
condition, the  work is then  to  prove that this
decay condition is preserved under inversion.
Our goal is more ambitious: we aim at a  systematic construction of
decay properties that are inherited by matrix inversion. 

Our  tools are borrowed from approximation theory and from the theory
of operator algebras and Banach algebras.  At first glance, these tools  have nothing to
do with the \odd\ of  infinite matrices. The appearance of operator algebras
becomes plausible when we observe that almost all known  conditions for
off-diagonal decay define Banach algebras. The appearance of
approximation theory (in the form of smoothness spaces and
approximation spaces) is perhaps more surprising. Indeed,  the connection between the
problem of \odd\ and approximation theory is one of the main insights
of this paper. 

To put our problem into an abstract setting, suppose that we are given
a Banach algebra $\mB $. We may think of $\mB $ as an algebra of
infinite matrices whose norm describes some form of \odd . We first try
to find a systematic construction of subalgebras $\mA \subseteq \mB$ such that an
element $a\in \mA $ is invertible in $\mA $ if and only if $a$ is
invertible in the larger algebra $\mB $. In the context of matrices, we think of the
smaller algebra as an algebra describing a stronger decay condition. 
Technically, we say that a unital Banach algebra $\mA $ is \emph{\IC } in
$\mB $, if $a\in \mA $ and $a^{-1} \in \mB $ implies $a^{-1} \in \mA
$. Inverse-closedness occurs under various names (spectral invariance,
Wiener pair, local subalgebra, etc.) in many fields of mathematics,
see the survey~\cite{Gro09}. While often the existence of an
inverse-closed subalgebra is taken for granted, e.g., in
non-commutative geometry~\cite{ Black98, connes}, our interest is in the systematic
construction of inverse-closed subalgebras and their application to
matrix algebras.

We present and investigate two constructions of \IC\ subalgebras of a
Banach algebra, both of which are inspired by ideas from approximation
theory. 

The first idea is the construction of ``smooth'' subalgebras via
derivations. In
its essence this idea is a generalization of the quotient rule for the
derivative. If a continuously differentiable function $f$  on an interval
does not have any zeros, then its inverse $1/f$ is again continuously
differentiable. In the abstract context of a \BA\ $\mA $ the
derivative  is replaced by an unbounded  derivation. Then  the same
algebraic 
manipulations used to  prove the quotient rule show   that the domain of a
derivation is \IC\ in the original algebra. The first result of this
type is due  to
the fundamental work of Bratteli and Robinson~\cite{bratteli75,bratteli76} and can also be found in Connes~\cite{connes}.
For the case of a matrix algebra over $\bbZ $  the relevant derivation is the
commutator with the diagonal matrix $X$ with entries $X(k,l) =
2 \pi i k\delta _{k,l}, k,l \in \bbZ $. The commutator $A \mapsto [X,A]= XA-AX$ is a derivation and
$[X,A]$ has the entries $[X,A](k,l) = 2 \pi i (k-l) A(k,l), k,l\in \bbZ
$.  Clearly, if $[X,A]$ enjoys some \odd , then $A$ itself
has a better \odd .  (Such commutators are used implicitly in Jaffard's
work~\cite{Jaffard90}.)

To demonstrate how far this idea can be pushed we formulate
explicitly a multivariate statement with anisotropic \odd\
conditions. 

\begin{thm}
 Let $A$ be a matrix over $\bbZd$, $r>d$, and $\alpha=(\alpha_1, \dots , \alpha_d)
 \in \bbZd$ with $\alpha_j\geq 0$. If $A$ is invertible on $\ell ^2(\bbZd)$
 and satisfies the anisotropic off-diagonal decay condition
 \begin{equation}
   \label{eq:ch9}
   |A(k,l)| \leq C (1+|k-l|)^{-r} \prod _{j=1}^d (1+|k_j-l_j|)^{-\alpha_j}
   \qquad k,l \in \bbZd \, ,
 \end{equation}
then the entries of the inverse matrix $A^{-1}$ satisfy an estimate of
the same type
$$
|(A^{-1}(k,l)| \leq C' (1+|k-l|)^{-r} \prod _{j=1}^d
(1+|k_j-l_j|)^{-\alpha_j} , 
   \qquad k,l \in \bbZd \, . 
$$
\end{thm}

 To fill in the gap between integer rates of decay
(corresponding to $C^k$-functions), we turn to  \AT , which
suggests the concept of 
fractional smoothness and offers the  H\"older-Zygmund spaces. To
treat these constructions in the general context of \BA s, we need
more structure and consider \BA s with a  commutative
automorphism group and the associated generators.

The second idea for the construction of \IC\ matrix algebras is  based
on the intuition that a 
matrix with fast \odd\ can be approximated well by banded
matrices. Approximation theory  offers the concept of approximation spaces in order to
quantify the rate of approximation. It is therefore natural to study
the approximation properties  of matrices by banded matrices. As a
sample result we quote the following statement (cf. Corollary~\ref{cor:apprmas}).
\begin{thm} \label{thm-intro-matic}
  For a matrix $A= (A(k,l))_{k,l\in \bbZ }$   let $A_N$ be the banded
approximation of $A$ of width $N$ defined by the entries $A_N(k,l) =
A(k,l)$ for $|k-l|\leq N$ and $A_N(k,l) = 0$ otherwise. If $A$ is
invertible on $\ell ^2(\bbZ )$ and for some constants $r,C>0$
$$ 
\|A - A_N \|_{\ell ^2 \to \ell ^2} \leq C N^{-r}  \qquad \text{ for
  all } N > 0, 
$$ 
then there exists a sequence $B_N$ of banded matrices of width $N$,
such that 
$$
 \|A^{-1} - B_N \|_{\ell ^2 \to \ell ^2} \leq C' N^{-r}  \qquad \text{ for
  all } N > 0 \, .
$$
\end{thm}

For general \BA s one needs a substitute for the banded matrices and
postulates the existence of a graded sequence of subspaces compatible
with the algebraic structure. Then  one may define
approximation spaces and show that they form \IC\ subalgebras of the
original algebra. This line of research has been started in~\cite{Almira01,Almira06}.
  As an application
to operator algebras we construct a  new class of  \IC\ subalgebras
in ultra hyperfinite (UHF) algebras.

A further inspiration taken from  \AT\ is the equivalence of smoothness and
approximability. This principle is omnipresent in \AT\ and we
will add another facet to it. We will prove a 
Jackson-Bernstein theorem that is valid for a general \BA\ with a
commutative automorphism group. We will  show that the construction of
H\"older-Zygmund spaces via derivations and the approximation
properties based on ``bandlimited'' elements are equivalent and lead
to the same algebras. In other words, in a  \BA\  with enough structure
the two construction principles (smooth subalgebras and approximation
spaces) coincide. For matrix algebras we may rephrase the
fundamental paradigm of \AT\ by saying that  the approximability of
matrices by banded matrices is equivalent to off-diagonal decay. Thus the \odd\ of a
matrix describes some form of smoothness.

The abstract \BA\ methods not only yield  an
elegant explanation of some known results, but, more importantly,  we also
obtain new forms of \odd . 
Who would have guessed the following \odd\ condition? We certainly did
not, but derived it from the general theory. 

\begin{thm} \label{thm:baskalgapp}
Assume that $A= (A(k,l))_{k,l \in \bbZ } $  satisfies the following
\odd\ for some $r>0$:
  \begin{equation}
    \sup _{k\in \bbZ } |A(k,k)| <\infty, \quad
    2^{k r}\sum_{2^{k}\leq \abs{l}< 2^{k+1} }  \sup _{m\in
      \bbZ} |A(m,m-l)|  \leq C \quad \text{for all } k\geq 0. 
  \end{equation}
If $A$ is invertible on $\ell ^2(\bbZ )$, then the inverse matrix
satisfies the same form of \odd . 
\end{thm}

In our presentation we will argue on three levels. The first level is classical approximation theory, which will
serve us as a motivation. We then switch to the level of abstract
\BA s and define the concepts and tools required for the investigation of
non-commutative approximation theory in \BA s. Finally we 
return to the level of matrix algebras and express the abstract results as
statements about the \odd\ of matrices. 
Our main interest lies in the algebra properties and the invertibility in such spaces. These aspects are rarely addressed in approximation theory.

An operator algebraist will probably not find a new result in
Section~\ref{sec:smooth}, and an approximation theorist will be quite familiar
with the machinery of approximation spaces. However, an operator
algebraist will learn that methods from approximation theory yield new
constructions for smooth subalgebras (and thus something new might be
gained for non-commutative geometry). An  approximation theorist 
might find  a new playground ahead, namely the approximation theory in
operator algebras. 

\emph{Outlook.} Once the connection between inverse-closedness and
approximation theory is understood, one may exploit the entire arsenal
of approximation theory to obtain new constructions of smooth
subalgebras that are inverse-closed in the original algebra. In
particular, for matrices one may  define off-diagonal decay conditions
that amount to Besov smoothness or to
quasi-analyticity~\cite{Klo09}. Such refinements will be the subject
of forthcoming  work. \\

The paper is organized as follows: In Section~2 we collect the
resources from the theory of Banach algebras. In Section~3 we study
concepts of smoothness in general \BA s.  These are based
on the existence of suitable unbounded derivations or on the presence
of $d$-parameter automorphism groups. In Section~4 we pursue the idea
of approximation spaces attached to a \BA\ and study the quantitative
approximation of matrices by banded matrices. Finally, in  Section~5
we prove a theorem of Jackson-Bernstein type and obtain a completely
new set of \odd\ conditions.

\textbf{Acknowledgement:} We would like to thank Palle Jorgensen who
suggested a possible connection between Jaffard's result and the
theory of derivations to one of us  already many years ago. 

A preliminary version of our results was already announced in
``Oberwolfach Reports'', vol. 4(3) (2007), pp.\ 2079. 

\section{Resources}
\label{sec:ressources}
\subsection{Notation}
\label{sec:notions}
For $x$ in \bbRd\ let $\abs{x}$ denote be the $1$-norm of $x$, $\abs{x}_2$ the $2$-norm, and $\abs x_\infty$ the
sup-norm.  The vectors $e_k$, $1 \leq k \leq d$, denote the standard
basis of \bbRd.  A multi-index $\alpha= (\alpha _1, \dots , \alpha _d)
\in \mathbb{N}_0^d$ is a  
$d$-tuple of nonnegative integers. We set
$x^\alpha=x_1^{\alpha_1}\cdots x_d^{\alpha_d}$,  and  $ D^\alpha
f(x)=\frac {\partial^{\alpha_1}}{\partial x_1^{\alpha_1}} \cdots \frac
{\partial^{\alpha_d}}{\partial   x_d^{\alpha_d}} f (x) $ is the
classical partial  derivative.  The degree  of $x^\alpha$ is
$\abs{\alpha} = \sum _{j=1}^d \alpha _j$, and $\beta \leq \alpha $
means that $\beta _j \leq \alpha _j$ for $j=1, \dots ,d$. 

Let $S(\bbRd)$ denote the Schwartz space of rapidly decreasing functions on \bbRd. The Fourier transform of $f
\in S(\bbRd)$ is $\mF f (\omega) = \int_{\bbRd} f(x) e^{-2 \pi i \omega \cdot x} \; dx$. This definition is
extended by duality to $S'(\bbRd)$, the space of tempered
distributions. 

A submultiplicative weight on \bbRd\ (or on $\bbZd$) is a positive function $v:\bbRd \to \bbR$ such that $v(0)=1$
and 
$v(x+y) \leq v(x)v(y)$ for $x,y \in \bbRd$.  The standard submultiplicative weights are the polynomial weights
$v_m(x) = (1+\abs x)^m$ for $m\geq 0$.

 The notation $ f \asymp g$
 means there are constants $C_1, C_2 > 0$ such that $C_1 f \leq g \leq C_2 f$.
Here $f$ and $g$ are two positive functions depending on other
parameters. 
Banach spaces with equivalent
norms are considered as equal. 

The operator norm of a bounded linear mapping $A: X \to Y$ between Banach spaces is denoted by $\norm{A}_{X \to Y}$.
\subsection{Concepts from the Theory of Banach Algebras}
\label{sec:concepts-from-theory}
Besides standard notions from \BA\ theory we will use some less known
concepts. 

  All \BA s will be \emph{unital}. To verify that a \BS\ \mA\ with norm $\norm{\phantom{i}}_\mA$ is a \BA\ we will often prove the weaker property $\norm{ab}_\mA \leq C \norm{a}_\mA \norm{b}_\mA$ for some constant $C$.
The norm $\norm{a}'_\mA=\sup_{\norm{b}_\mA =1}\norm{ab}_\mA$ is then an equivalent norm on \mA\ and satisfies $\norm{ab}'_\mA \leq \norm{a}'_\mA \norm{b}'_\mA $.

\begin{defn}[Inverse-closedness]
  Let $\mA \subseteq \mB$ be  a nested pair of \BA s with a common identity. Then \mA\ is called
  \emph{inverse-closed} in \mB, if
  \begin{equation} \label{eq_3} a \in \mA \text{ and } a^{-1} \in \mB \quad \text{implies}\quad a^{-1} \in \mA.
  \end{equation}
\end{defn}
Inverse-closedness is equivalent to \emph {spectral invariance}. This
means  that  the spectrum 
$ \sigma_\mA(a) =\set{\lambda \in \bbC \colon a-\lambda \text{ not invertible in } \mA} $ of an element  $a \in
\mA$  does not depend on the algebra and so 
\begin{equation*}
  \sigma_\mA(a)= \sigma_\mB (a), \quad \quad \text{ for all } \, a \in \mA.
\end{equation*}
If \mA\ is \IC\ in \mB\, and \mB\ is \IC\ in \mC, then \mA\ is \IC\ in \mC.

\subsubsection*{The Lemma of Hulanicki}
\label{sec:lemma-hulanicki}
The verification of inverse-closedness is often nontrivial. Under
additional conditions this verification is 
sometimes possible by using an argument of Hulanicki~\cite{Hul72, FGL06}.

 Recall that a Banach $*$-algebra is \emph{symmetric},  if the spectrum of positive elements is
non-negative, $\sigma _{\mA } (a^*a)\subseteq [0,\infty)$ for all $a \in\mA$.
Denote the spectral radius of $a \in \mA$ as $\rho_\mA(a)=\sup \set{\abs \lambda : \lambda \in \sigma_\mA(a)}$ .
\begin{prop}[Hulanicki's Lemma] \label{prop:lemma-hulanicki} Let \mB\ be a symmetric \BA, $\mA
  \subseteq \mB$ a $*$-subalgebra with common involution and common  unit element. \tfae
  \begin{enumerate}
  \item \mA\ is \IC\ in \mB .
  \item $\rho_\mA(a)= \rho_\mB(a) \text{ for all } a=a^*$ in \mA.
  \end{enumerate}
\end{prop}
In particular, if \mA\ is a closed $*$-subalgebra of \mB , then \mA\ is \IC\ in \mB.

\subsubsection*{Brandenburg's trick \cite{Brandenburg75}}
\label{sec:method-brandenburg}
This method is sometimes  used to prove the equality of spectral radii. Let $\mA\subseteq\mB$ be two \BA s with the
same identity. Assume that the norms satisfy
\begin{equation}
  \label{eq:25}
  \norm{ab}_{\mA} \leq C
  (\norm{a}_{\mA}\norm{b}_{\mB}+\norm{b}_{\mA}\norm{a}_{\mB}) \qquad  \text{ for all } a,b \in \mA.
\end{equation}
Applying (\ref{eq:25}) with $a=b=c^n$ yields
\begin{equation*}
  \norm{c^{2n}}_\mA \leq 2C \norm{c^n}_\mA \norm{c^n}_\mB.
\end{equation*}
Taking $n$-th roots and the limit $n \to \infty$ gives $ \rho_\mA(c) \leq \rho_\mB(c)$. Since the reverse
inequality is always true for $\mA \subseteq \mB$, we obtain the equality of spectral radii. By
Proposition~\ref{prop:lemma-hulanicki} $\mA $ is inverse-closed in $\mB $.

\subsection{Examples of Smoothness and Matrix Algebras}
\label{sec:standard-examples}
We will use two classes of examples. The smoothness spaces $C^k$ and
$\Lambda_r$ on \bbRd\ serve to motivate some 
abstract concepts, and the  matrix algebras serve as the fundamental
Banach algebras to which we will apply the general theory.

\subsubsection*{Smoothness Spaces}
These are the spaces $C^k(\bbRd)$ with norms
\begin{equation}
  \norm{f}_{C^k}=\sum_{\abs \alpha \le k}\norm{D^\alpha f}_\infty.
\end{equation}
Using the translation operator $T_t$,  $T_tf(x)=f(x-t), t,x\in \bbRd$, the H\"older-Zygmund spaces $\Lambda_r(\bbRd)$
are defined with the help of the seminorms
\begin{equation}
  \label{lip*}
  \abs{f}_{\Lambda_r}= \sup_{\abs t \neq 0 } \abs{t}^{-r}{\norm{T_tf - 2 f +  T_{-t}f}_\infty},\quad 0 < r \leq 1  .
\end{equation}
For $r=k+\eta$, $k \in \bbN_0$, $0< \eta \le1$ the norm
\begin{equation}
  \label{eq:holderzygmund}
  \norm{f}_{\Lambda_r} = \norm{f}_{C^k}+ \max \set{\abs{D^\alpha f}_{\Lambda_\eta} \colon \abs \alpha =k}
\end{equation}
defines the H\"older-Zygmund space $\Lambda_r (\bbRd) $ of order $r$.
\subsubsection*{Matrix Algebras}
\label{sec:matalg}
One of the main insights of this paper is the striking similarity between trigonometric approximation and
approximation of matrices by banded matrices. To describe the most common forms of off-diagonal decay, let us fix some
notation. An infinite matrix $A$ over \bbZd\ is a function $A:\bbZd \times \bbZd \to \bbC$.  The $m$-th side
diagonal of $A$ is the matrix $\hat A (m)$ given by
\begin{equation}
  \label{eq:8}
  \hat A(m)(k,l)=\begin{cases}
    A(k,l),& \quad k-l=m,\\
    0,     & \quad \text{otherwise}.
  \end{cases}
\end{equation}
With this notation a matrix $A$ is \emph{banded} with bandwidth $N$, if
\begin{equation}
  \label{eq:bandmatrix}
  A=\sum_{\abs m \leq N} \hat A (m).
\end{equation}
Let us define the most common examples of matrix algebras over \bbZd.
\\
The \emph{Jaffard algebra} $\mJ_r$, $ r>d$, is defined by the norm
\begin{equation}
  \label{eq:jaffnorm}
  \norm{A}_{\mJ_r}=\sup_{k,l \in \bbZd}\abs{A(k,l)}v_r(k-l), \quad r>d.
\end{equation}
Explicitly, $A \in \mJ_r \Leftrightarrow |A(k,l)| \leq C (1+|k-l|)^{-r} $, so the norm of $\mJ_r$ describes
polynomial decay off the diagonal.  Writing the norm in terms of the
side-diagonals, we will often use that
\begin{equation}
  \label{eq:ch2}
  \|A\|_{\mJ_r} = \sup _{k\in \bbZ ^d} \|\hat{A}(k)\|_{\ell^2 \to \ell^2}  (1+\abs{k})^{r} \, .  
\end{equation}
The \emph{algebra of convolution-dominated matrices} $\mC_r, r\geq0$,  (sometimes called the
Baskakov-Gohberg-Sj\"ostrand algebra) consists of all matrices $A$, such that the norm
\begin{equation}
  \label{eq:basknorm}    
  \norm{A}_{\mC_r}= \sum_{k \in \bbZd}\sup_{l \in \bbZd} \abs{A(l,l-k)} v_r(k) = \sum _{k\in
    \bbZd} \|\hat A (k) \|_{\ell^2 \to \ell^2} \, (1+|k|)^r 
\end{equation}
is finite. This is the weighted $\ell^1$-norm of the suprema of the side diagonals.
\\
The \emph{Schur algebra} $\mS_r$, $r\geq0$, is defined by the norm
\begin{equation}
  \label{eq:schurnorm}
  \norm{A}_{\mS_r} = \max \Big\{ \sup_{k\in \bbZd} \sum_{l\in \bbZd}  \abs{A(k,l)}v_r(k-l),
  \sup _{l\in \bbZd}\sum_{k\in \bbZd} \abs{A(k,l)}v_r(k-l)\Big\}.
\end{equation}

We note that the norms above depend only on the absolute values of the matrix entries. Precisely, we say that a
matrix norm on $\mA $ is \emph{solid}, if $B \in \mA$ and $\abs{A(k,l)} \leq \abs{B(k,l)}$ for all $k,l$ implies
$A \in \mA$ and $\norm{A}_{\mA}\leq\norm{B}_{\mA}$.  

The following result summarizes the main properties of the
matrix classes $\mC , \mJ, \mS $. See \cite{Bas90, GL04a,Jaffard90} for proofs.



\begin{prop} \label{fundamental} Let $\mA $ be one of the matrix
  classes   $\mJ _r$ for $r>d$, $\mS _r $ for
  $r>0$, and $\mC _r$ for $r\geq 0$. 

(i) Then   $\mA $ is  a solid Banach $*$-algebra  with respect to
matrix multiplication 
  and taking adjoints as the involution. 

(ii) Every  $\mA $  is continuously embedded into $\mB (\ell ^p(\bbZd))
$,$ 1\leq p \leq\infty$.

(iii)  Every  $\mA $ is inverse-closed in $\bopzd $. In particular, $\mA $ is symmetric. 
\end{prop}

  In the sequel we will construct algebras that are inverse-closed in one of the standard
algebras $\mC _r, \mJ _r, \mS _r$. These ``derived'' algebras will then be automatically inverse-closed in
$\bop$. In this sense Proposition~\ref{fundamental} is fundamental.

For a more general description of the \odd\ we use weight functions. Let $\mA $ be a matrix algebra and $v$ an
even weight function on $\bbZ ^d$. Then
\begin{equation} \label{eq:1a} \mA_v=\set{A \in \mX : \bigl(A(k,l)v(k-l)\bigr)_{k,l, \in \bbZd} \in \mA}.
\end{equation}
Writing $\tilde A (k,l)=A(k,l) v(k-l)$, the norm on $\mA _v$ is given by $\|A\|_{\mA _v} = \|\tilde{A} \|_{\mA
}$.

With this definition, the standard matrix algebras of Proposition~\ref{fundamental} are just weighted versions of
the basic types $\mC, \mJ$, and $\mS $. Specifically, using the polynomial weight $v_r(k) = (1+|k|)^r$, we have
\begin{equation}
  \label{eq:ch8}
  \mC _r =   (\mC _0)_{v_r}, \,\, \mS _r =   (\mS _0)_{v_r}, \,\, \text{
    and  } \,\, \mJ _{s +r} = (\mJ _s ) _{v_r} \, . 
\end{equation}

\begin{prop}
  \label{prop:solid-ma}
  If \mA\ is a solid \MA\ and $v$ is a submultiplicative weight, then $\mA_v$ is a solid \MA.
\end{prop}
\begin{proof}
  The only nontrivial part is to verify that the $\mA_v$-norm is submultiplicative. Let $A,B$ be in $\mA_v$. We
  write $\tilde A (k,l)=A(k,l) v(k-l)$ and $\abs A$ for the matrix with entries $\abs{A(k,l)}$, then
  \begin{equation*}
    \begin{split}
      \abs{\widetilde{AB} }(k,l) &= \Bigl \lvert {\sum_m A(k,m)B(m,l)v(k-l)} \Bigr \rvert \\
      &\leq \sum_m \abs{A(k,m)v(k-m)} \; \abs{B(m,l)v(m-l)}=(\abs{\tilde A }\abs{\tilde B })(k,l)\, .
    \end{split}
  \end{equation*}
  Consequently, $\|AB\|_{\mA _v} = \|\widetilde{AB}\|_{\mA } \leq \|\widetilde{A}\widetilde{B}\|_{\mA }\leq \|\widetilde{A}\|_{\mA } \|\widetilde{B}\|_{\mA
  } = \|A\|_{\mA _v} \|B\|_{\mA _v} $.
\end{proof}
We do not know if the proposition remains true for non solid \MA s.
\section{Smoothness in Banach Algebras}
\label{sec:smooth}
In classical analysis the smoothness is measured by derivatives and by higher order difference operators. In this
section we the identify corresponding structures for \BA s: these are unbounded derivations, groups of
automorphisms, and algebra-valued H\"older spaces. The standard literature (e.g. \cite{bratteli96, bratrob87}) is
formulated for $C^*$-algebras and densely defined derivations, whereas we work mostly with Banach $*$-algebras
and derivations without dense domain. We are therefore obliged to be especially careful before adopting a result
for our purpose and will provide streamlined proofs where necessary.

\subsection{Derivations}
\label{sec:derivations}
For real functions the smoothness is measured by derivatives. The corresponding concept for \BA s are unbounded
derivations.  A \emph{derivation} $\delta$ on a \BA\ \mA\ is a closed linear mapping $\delta \colon \mD \to \mA$,
where the \emph{domain} $\mD=\mD(\delta)=\mD(\delta,\mA)$ is a (not necessarily closed or dense) subspace of \mA,
and $\delta$ fulfills the Leibniz rule
\begin{equation}
  \label{eq:derivation}
  \delta(ab) =a \delta(b) + \delta(a) b \qquad \text{for all} \,\,
  a,b \in \domd. 
\end{equation}
If \mA\ possesses an involution, we assume that the derivation and the
domain are  symmetric, i.e., $\mD=\mD^*$ and
$\delta(a^*) =\delta(a)^*$ for all $a \in \mD$. The domain is normed with the graph norm
$\norm{a}_{\domd}=\norm{a}_\mA+\norm{\delta(a)}_\mA$.
Equation~(\ref{eq:derivation}) implies that $\domd$ is a (not necessarily unital) Banach algebra, and the
canonical mapping $\domd \to \mA$ is a \cont\ embedding.
\begin{ex} [Derivations on $L^\infty$]
  The classical derivative $\frac{d}{dx}: f\mapsto f'$ is a closed, symmetric derivation on the von Neumann-algebra
  $L^\infty(\bbR)$. The domain of $\frac{d}{dx}$ in $L^\infty(\bbR)$ consists of all Lipschitz functions with
  essentially bounded derivative.
  Clearly, $\mD(\delta,L^\infty)$ is not dense in $L^\infty$.
\end{ex}
\begin{ex}[Derivations on Matrix Algebras] \label{ex:derivations-MAs} Let \mA\ be a \MA\ over \bbZ.  Define the
  diagonal matrix $X$ by $X(k,k)= 2 \pi i k$. Then the formal commutator
$$
\delta_X(A)=[X,A]=XA-AX
$$
has the entries $ [X,A](k,l)=2 \pi i (k-l)A(k,l)$ for $ k,l \in \bbZ $, and $\delta _X$ defines a closed,
symmetric derivation on \mA.

This derivation is closely related to the weighted matrix algebra
$A_{v_1}$, at least for \emph{solid} \MA s. 
\begin{prop} \label{prop:solid_MA_domain} Let \mA\ be a solid \MA\ over \bbZ. Then $\mD(\delta,\mA)=\mA_{v_1}$,
  and the norms $\norm{\phantom{I}}_{\domd}$ and $ \norm{\phantom{I}}_{\mA_{v_1}}$ are
  equivalent.  
\end{prop}
\begin{proof}
  Recall that $\tilde A (k,l) = A(k,l) v_1(k-l) = A(k,l) (1+|k-l|)$. Since the norm of $A\in \mA$ depends only on
  the absolute values of the entries of $A$, we obtain that
$$
\|A\|_{\mA _{v_1}} = \|\tilde A \|_{\mA } \leq \|A \|_{\mA } + \|\, [X,A] \|_{\mA } = \|A \|_{\mD (\delta )} \leq
2\cdot 2 \pi  \|A\|_{v_1} \, ,
$$
as claimed.
\end{proof}
\end{ex}
If $\delta$ is a densely defined $*$-derivation of a $C^*$-algebra \mA , then by a result in ~\cite{bratteli75}
$\one \in \domd$ and $\domd$ is \IC\ in \mA.  In \cite{MR1221047} this result was extended to densely defined
derivations on arbitrary \BA s without involution structure. We need
an extension for derivations that are not necessarily densely defined.
\begin{thm} \label{thm:derivationsIC} Let \mA\ be a symmetric \BA, and
  $\delta$ a symmetric  derivation on
  \mA. 
  If $\one \in \domd$, then $\domd$ is \IC\ in \mA\ and $\domd$ is a
  symmetric \BA .  Then the quotient rule
  \[
  \delta(a^{-1})=- a^{-1}\delta(a)a^{-1}
  \]
  is valid, and yields the explicit norm estimate
  \[
  \norm{a^{-1}}_{\mD(\delta)} \leq \norm{a^{-1}}^2_\mA \norm{a}_{\mD(\delta)}.
  \]
\end{thm}
\begin{proof}
  The proof in \cite{bratteli75} uses functional calculus and could be adapted to the setting of the theorem. We
  prefer a short conceptual argument based on Hulanicki's Lemma (Proposition~\ref{prop:lemma-hulanicki}). We show
  that $\rho_{\domd}(a)=\rho_{\mA}(a)$ for any $a=a^*$ in $\domd$. Using the inequality
  \[
  \norm{\delta (a^n)}_\mA \leq n \norm{a}^{n-1}_\mA \norm{\delta(a)}_\mA
  \]
  which can be established by induction, we estimate the norm of $a^n$ by
  \begin{equation*}
    \norm{a^n}_{\domd}=\norm{a^n}_\mA+\norm{\delta(a^n)}_\mA \leq \norm{a}_\mA^n+n \norm{a}_\mA^{n-1}\norm{\delta (a)}_\mA.
  \end{equation*}
  Taking $n$-th roots on both sides and letting $n$ go to infinity, we obtain $ \rho_{\domd}(a) \le
  \norm{a}_\mA$, and consequently $ \rho_{\domd}(a) \le \rho_\mA(a)$. The reverse inequality $ \rho_{\mA}(a) \leq
  \rho_{\domd}(a)$ is always true for \BA s, since $\mD (\delta , \mA
  )\subseteq \mA $,  so Proposition~\ref{prop:lemma-hulanicki} implies
  that $\mD(\delta)$ is  \IC\ in \mA. Consequently $\sigma _{\domd }
  (a^*a) = \sigma _{\mA }(a^*a)\subseteq [0,\infty )$ for all $a\in
  \domd  $,  and thus $\domd $ is a symmetric \BA . 

  Thus, if $a\in \mD (\delta )$ and $a^{-1} \in \mA $, then $a^{-1}
  \in \mD (\delta )$ and so $\delta (a^{-1})$ is well-defined in
  $\mA$. Therefore  the 
  quotient rule and the norm inequality follow from  the Leibniz
  rule $0
  =\delta(\one)=\delta(a a^{-1})= \delta (a) a^{-1} + a \delta
  (a^{-1})$. 
\end{proof}
\begin{rems}
  Theorem~\ref{thm:derivationsIC} is remarkable because it yields an explicit norm control of the inverse in the
  subalgebra $\domd$. Results of this type are very  rare, see
  \cite{Nikolski99} for typical no-go results.
\end{rems}

\subsection*{Commuting Derivations}
\label{sec:comm-deriv}
The formulation of \IC ness results for matrices over \bbZd ,   and the
definition of higher orders of smoothness 
require  derivations for each ``dimension'' of the index set \bbZd.  

Let $\{\delta_1,\cdots, \delta_d\}$ be a set of commuting derivations
on the \BA\ \mA. Since products of unbounded operators and their
domains are a subtle and rather technical subject with many
pathologies, we will make the following  assumptions and thus
avoid many technicalities.

The domain of a finite product $\delta _{r_1} \delta _{r_2} \dots
\delta _{r_n}$, $1\leq r_j \leq d$ is defined by induction as 
$$
\mD (\delta _{r _1} \delta _{r _2} 
\dots \delta _{r _n}) = \mD (\delta _{r _1}, \mD (\delta _{r _2} 
\dots \delta _{r _n})) \, .
$$

We will assume througout that the operator $\delta _{r_1} \delta _{r_2} \dots
\delta _{r_n}$ and its  domain \linebreak[4] $\mD (\delta _{r _1} \delta _{r _2} 
\dots \delta _{r _n})$ are independent of the order of the factors
$\delta _{r_j}$. 

Then for every multi-index $\alpha$ the operator $\delta^\alpha
=\prod_{1 \leq k \leq d}\delta_k^{\alpha_k}$ and 
its domain $\mD(\delta^\alpha)$ are well defined.
  In analogy to $C^k(\bbRd)$ we equip $\mD(\delta^\alpha)$ with the norm
$$
\norm{a}_{\mD(\delta^\alpha)} =  \sum_{ \beta \leq 
  \alpha}\norm{\delta^\beta(a)}_\mA\, .$$

Since $\delta _j$ is assumed to be a closed operator on $\mA $, it
follows that $\delta _j$ is a closed operator on $\mD (\delta ^\alpha
)$. 


\begin{defn}
  Let \mA\ be a \BA\ and $k$ a nonnegative integer. 
  The \emph{derived space of order k} is
  \begin{equation*}
    \mA^{(k)}=\bigcap_{\abs{\alpha} \leq k} \mD(\delta^\alpha),
    \qquad \text{ and }  \qquad     \mA^{(\infty)}=\bigcap_{k=0}^\infty \mA^{(k)}.
  \end{equation*}
\end{defn}
We summarize the results on commuting derivations.
\begin{lem}
  Let $\{\delta_k : 1 \leq k \leq d\}$, be a set of commuting derivations on the \BA\ \mA.
\begin{enumerate}
  \item[(i)] Then $\mD(\delta^\alpha)$ is a (not necessarily unital)
    subalgebra of \mA\ for every  $\alpha \in \mathbb{N}_0^d$.

  \item[(ii)] Let $\mathcal{R} \subseteq \bbNd _0$ be an arbitrary finite index set and set
  \[
  \mD_{\mathcal{R}}(\delta)=\bigcap_{\alpha \in \mathcal{R}} \mD(\delta^\alpha) \, .
  \]
  Then $\mD_{\mathcal{R}}(\delta) $ is a Banach-subalgebra of \mA\ with the norm 
  $\norm{a}_{\mD_{\mathcal{R}}(\delta)}= \sum\limits_{\alpha \in \mathcal{R}}\norm{a}_{\mD(\delta^\alpha)}$. In
  particular $\mA^{(k)}$ is a Banach-subalgebra of \mA.
\end{enumerate}
\end{lem}
  \begin{proof}
We first remark that the Leibniz rule \eqref{eq:derivation} implies
the general Leibniz rule
\begin{equation}
  \label{leib}
  \delta ^\alpha (ab) = \sum _{\beta \leq \alpha }
  \binom{\alpha}{\beta} \delta ^\beta (a) \delta ^{\alpha-\beta } (b)
  \, .
\end{equation}
If $a,b \in \mD (\delta ^\alpha )$, i.e., $\delta ^\beta (a), \delta
^\beta (b) \in \mA $ for $\beta \leq \alpha $, then clearly $ab \in
\mD (\delta ^\alpha )$ and the norm inequality $\|ab\|_{\mD (\delta
  ^\alpha )} \leq C \|a\|_{\mD (\delta   ^\alpha )} \|b\|_{\mD (\delta
  ^\alpha )} $ follows  after taking norms in \eqref{leib}. Since the
finite intersection of \BA s is a \BA , $\mA ^{(k)}$ and $\mD
_{\mathcal{R} }(\delta )$ are \BA s.  
  \end{proof}
  \begin{prop} \label{prop_comm-deriv} Assume that \mA\ is a symmetric
    Banach algebra with a set of commuting 
    symmetric derivations $\{ \delta _k : 1\leq k \leq d\} $ satisfying $\one \in \mD(\delta_k), 1\leq k \leq d$.
    Then $\mD(\delta ^\alpha )$ is \IC\ in \mA.  Furthermore, the \BA\ $\mD_{\mathcal{R}}(\delta)$ is \IC\
    in \mA , and $\mA^{(\infty)}$ is a Fr\'{e}chet algebra that is \IC\ in \mA.
  \end{prop}
  \begin{proof}
    Let $\delta^\alpha = \delta_{r_n} \dotsm \delta_{r_1}$ with  $n=\abs
    \alpha$ and  $1 \leq r_j \leq d$ for all 
    $j$. By Theorem~\ref{thm:derivationsIC} $\mD (\delta _1, \mA )$ is
    a symmetric \BA\ and inverse-closed in $\mA $. Now we argue by
    induction and assume that $\mD (\delta _{r_j} \dots \delta
    _{r_1})$ is symmetric and inverse-closed in  $\mA$.  Since by
    definition $\mD(\delta _{r_{j+1}} \dots 
    \delta _{1})=\mD(\delta_{r_{j+1}}, \mD (\delta _{r_j} \dots \delta
    _{r_{r_1}}))$ and $\delta _{r_{j+1}}$ is a  closed derivation
    on the symmetric \BA\  $ \mD(\delta_{r_j} \cdots
    \delta_{r_1})$, Theorem~\ref{thm:derivationsIC} asserts
    that $\mD(\delta _{r_{j+1}} \dots     \delta _{r_1})$ is symmetric
    and \IC\ in $\mD (\delta _{r_j} \dots \delta
    _{r_1})$ and thus \IC\ in $\mA $ by transitivity. 
We repeat this argument $n$ times and find  that $\mD (\delta ^\alpha )
= \mD (\delta _{r_n} \dots \delta _{r_1})$ is symmetric and 
\IC\ in $\mA $. 

Finally,  the finite or infinite intersection of inverse-closed
subalgebras of $\mA$ is again \IC\ in $\mA $. Specifically, if $a \in 
  \mD_{\mathcal{R}}(\delta)=\bigcap_{\alpha \in \mathcal{R}} \mD(\delta^\alpha)$
  and $a$ is invertible, then the argument above shows that $a^{-1}
  \in \mD (\delta ^\alpha, \mA )$ for each $\alpha \in \mathcal{R}$, whence
  $a^{-1} \in \mD_{\mathcal{R}}(\delta)$. The argument for $\mA
  ^{(\infty )}$ is the same. 
  \end{proof}

\begin{rem} The \IC ness of $\mA^{(\infty)}$ in \mA\ is implicit in~\cite{beals77}. \end{rem}
  \begin{ex} [Matrix algebras over \bbZd] \label{MAsonbzd}
    If \mA\ is a \MA\ over \bbZd, then we define the derivations $\delta_j(A)(k,l)= [X_j,A](k,l) = 2 \pi i
    (k_j-l_j)A(k,l) , 1\leq j \leq d$. These derivations are symmetric and commute with each other, and $\one \in
    \mD(\delta_j)$ for all $j$.  An application of Proposition ~\ref{prop_comm-deriv} gives that all spaces
    $\mD_{\mathcal{R}}(\delta)$ are \IC\ subalgebras of \mA.

    If \mA\ is solid there is an immediate generalization of Proposition \ref{prop:solid_MA_domain} to matrix
    algebras over the index set $\bbZd$.
    \begin{prop} \label{prop-solidmazd}
      Let \mA\ be a solid \MA\ over \bbZd. Then $\mA^{(m)}= \mA_{v_m}$ . In particular, $\mA_{v_m}$ is an
      \IC\ subalgebra of \mA.
    \end{prop}
    \begin{proof}
      The identity $\mA^{(m)} = \mA_{v_m}$ is proved as in Proposition~\ref{prop:solid_MA_domain}. The
      inverse-closedness follows from Proposition~\ref{prop_comm-deriv}.
    \end{proof}

    Using the characterization of the standard \MA s over \bbZd\ by weights (\ref{eq:ch8}) we spell out the preceeding result for these algebras.
    \begin{cor}
      For $k \in \bbN$ the algebra $\mC_k$ is \IC\ in $\mC_0$. Likewise, $\mS_k$ is \IC\ in $\mS_0$, and
      $\mJ_{r+k}$ is \IC\ in $\mJ_r$ for every $r>d$.
    \end{cor}

  \end{ex}

  The value of Proposition~\ref{prop_comm-deriv} lies in its potential
  to treat anisotropic decay conditions. As 
  an example we state the following anisotropic generalization of Jaffard's theorem.

\begin{prop}\label{jaffaniso}
  Let $A$ be a matrix over $\bbZd$, $r>d$, and $\alpha=(\alpha_1,
  \dots , \alpha_d) \in \bbNd _0$. If $A$ is
  invertible on $\ell ^2(\bbZd)$ and satisfies the anisotropic off-diagonal decay condition
  \begin{equation}
    \label{eq:ch9a}
    |A(k,l)| \leq C (1+|k-l|)^{-r} \prod _{j=1}^d
    (1+|k_j-l_j|)^{-\alpha_j}\, ,    \qquad k,l \in \bbZd \, ,
  \end{equation}
  then the entries of the inverse matrix $A^{-1}$ satisfy an estimate of the same type
$$
|(A^{-1}(k,l)| \leq C' (1+|k-l|)^{-r} \prod _{j=1}^d
(1+|k_j-l_j|)^{-\alpha_j}\, , \qquad k,l \in \bbZd \, .
$$
\end{prop}

\begin{proof}
 The off-diagonal decay condition is equivalent to saying that the matrix $\tilde A$ with entries $\tilde A(k,l)
  = \prod _{j=1}^d (k_j-l_j)^{\alpha_j} A(k,l)$ is in the Jaffard algebra $\mJ _r$.  But $\tilde A$ is just a multiple of $\prod
  _{j=1}^d \delta _j ^{\alpha_j} A = \delta ^\alpha A$, where $\delta_j(A) $ is defined in Example~\ref{MAsonbzd}. Since $\mD (\delta ^\alpha, \mJ_r)$ is inverse-closed in $\mJ _r$ by
  Proposition~\ref{prop_comm-deriv} and $\mJ _r$ is inverse-closed in $\bop$, $A^{-1}$ is again in $\mD (\delta
  ^\alpha, \mJ_r)$, which is nothing but the off-diagonal decay stated.
\end{proof}
\subsection{Automorphism Groups and Continuity}
\label{Automorphism groups and continuity}
Our next step is to treat the algebras $\mA_{v_r}$ with non-integer parameter $r$ in analogy to spaces with
fractional smoothness.  Two natural approaches are fractional powers of the generators or automorphism groups and
the associated H\"older-Zygmund continuity. We choose the latter approach and introduce a new structure, namely
automorphism groups. This choice is also  motivated by the failure to distinguish between the spaces $\mD(\frac{d}{dx},
L^\infty(\bbT))=\set{ f \in \Lip(\bbT): f' \in L^\infty(\bbT)}$ and $\mD(\frac{d}{dx}, C(\bbT))=C^1(\bbT)$ by means of
derivations alone. To explain this difference, we need to consider derivations that are generators of groups of
automorphisms.

An \emph{automorphism group}, more precisely a $d$-parameter
automorphism group acting on \mA ,  is a set of \BA\ automorphisms $\Psi=\set{\psi_t}_{t \in \bbRd}$ of
\mA\ with the group properties
\begin{equation}
  \psi_s \psi_t=\psi_{s+t} \quad \text{for all} \quad s,t \in \bbR^d.
\end{equation}
If \mA\ is a $*$-algebra, we assume that $\Psi$ consists of $*$-automorphisms.  In addition, we assume that
$\Psi$ is a \emph{uniformly bounded} automorphism group, that is,
\[
M_\Psi= \sup_{t\in \bbRd}\norm{\psi_t}_{\mA \to \mA} < \infty \, .
\]
This is all we need, but  clearly the abstract theory works for much more general group actions
~\cite{grushka07,Torba08}. 

An element $a$ of \mA\ is \emph{\cont}, if
\begin{equation}
  \label{strongcont}
  \norm{\psi_t(a)-a}_\mA \to 0 \text{ for } t \to 0.
\end{equation}
The set of \cont\ elements of \mA\ is denoted by $C(\mA)$.

\begin{ex}
  The classical example is the translation group $\set{T_x:x \in \bbRd}$.  
  For $\mA =L^\infty(\bbRd)$ the \cont\ elements are
  the functions in 
  $C(L^\infty(\bbRd)) = C_u(\bbRd)$, where $C_u(\bbRd)$
  denotes the space of bounded uniformly \cont\ functions on $\bbRd$.
\end{ex}

For $t \in \bbRd\setminus\set{0}$ the \emph{generator} $\delta_t$ is
\begin{equation}\label{eq:generator}
  \delta_t (a) = \lim_{h \to 0} \frac{\psi_{ht}(a)-a }{h}
\end{equation}
The domain of $\delta_t$ is the set of all $a \in \mA$ for which this limit exists. 
The \emph{canonical generators} of $\Psi$ are
$\delta_{e_k}$ and $\Psi$ is called the \emph{automorphism group generated by} $(\delta_{e_k})_{1\le k\le d}$.
Each generator $\delta_t$, $t \in \bbRd \setminus \set{0}$, is a closed derivation. If \mA\ is a
Banach $*$-algebra, then $\delta_t$ is a $*$-derivation~\cite{bratrob87}.

\begin{rems}
  (1) In a $C^*$-algebra all automorphisms are isometries. This is no longer true for symmetric algebras.

  (2) In the theory of operator algebras it is usually assumed that $\Psi $ is strongly continuous on all of \mA
  , i.e. $\mA = C(\mA )$.  This is no longer true for most  matrix algebras, and $C(\mA )$ is an
  interesting space in its own right. 
\end{rems}
\begin{defn}
  Let $M_t, t\in \bbR^d$, be the modulation operator $M_t x (k)= \cexp [k \cdot t] x(k)$, $ k\in \bbZd$.
  Then \begin{equation*} \chi_t(A) = M_t A M_{-t},\; \chi_t(A)(k,l)=\cexp [(k-l) \cdot t]A(k,l) \quad k,l \in
    \bbZd \, ,
  \end{equation*}
  defines a group action on matrices. 
\end{defn}
The derivations $\delta_k(A)=[X_k,A]$, $k=1, \dotsc , d,$ defined in Example \ref{MAsonbzd} are
just the canonical generators for the automorphism group $\chi$.
This automorphism group is uniformly bounded on 
each of the matrix algebras $\mJ_r, \mS_r, \mC_r$, and \bop, and on every solid \MA.

The following proposition states the \BA\ properties of $C(\mA)$.
\begin{prop} \label{prop:autom-groups-cont-IC} Let \mA\ be a \BA\  and $\Psi$ a uniformly bounded automorphism
  group acting on \mA. Then $C(\mA)$ is a closed and \IC\ subalgebra of $\mA$.  If \mA\ is a $*$-algebra, so is
  $C(\mA)$.
\end{prop}
\begin{proof}
  First we proof that $C(\mA)$ is an algebra. Let $a,b \in C(\mA)$. Then
  \begin{equation} \label{eq:CisIC} \norm{\psi_t(ab)-ab}_\mA \le \norm{\psi_t(a)}_\mA \norm{\psi_t(b)-b}_\mA +
    \norm{\psi_t(a)-a}_\mA \norm{b}_\mA.
  \end{equation}
  As $\norm{\psi_t}_{\mA \to \mA} \leq M_\Psi$ this expression tends to zero for $t\to 0$, so $ab \in C(\mA)$.
  For the completeness of $C(\mA)$ let $a_n \in C(\mA)$ for all $n$, and $a_n \to a$ in \mA. Then
  \begin{equation*}
    \norm{\psi_t (a)-a}_\mA \leq \norm{\psi_t (a-a_n)}_\mA + \norm{\psi_t (a_n)-a_n}_\mA + \norm{a_n-a}_\mA.
  \end{equation*}
  The first and the third term can be made arbitrarily small by choosing $n$ sufficiently large. Since $a_n \in
  C(\mA ) $, the second term can be made small.  Thus $a \in C(\mA)$.
  \\
  To show the \IC ness, let $a \in C(\mA)$ and assume that $a$ is invertible in \mA. Then (as in the proof of the
  quotient rule) the algebraic identity
  \begin{equation} \label{eq:Quotientrule} \psi_t(a^{-1})-a^{-1} = \psi_t(a^{-1}) (a-\psi_t(a)) a^{-1}
  \end{equation}
  yields that
  \begin{equation*}
    \norm{ \psi_t(a^{-1})-a^{-1}}_\mA \leq M_\Psi \,
    \norm{a^{-1}}_{\mA}^2\, \norm{a-\psi_t(a)}_\mA \to 0 \qquad \text{
      for } t \to 0, 
  \end{equation*}
  and thus $a^{-1} \in C(\mA)$.
\end{proof}

\subsection*{Generators and Smoothness}
Before defining the spaces $C^k(\mA)$, some technical preparations are needed, because generators commute only
under some additional conditions (similar to partial
derivatives).
\begin{prop}[~\cite{Butzer67, hille57}] \label{prop:density}
  \noindent \begin{itemize}
  \item[(i)]Let $\delta$ be the generator of a one-parameter group. Then the domain $\mD(\delta)$ is dense in
    $C(\mA)$.
  \item[(ii)]Let $\Psi$ be a $d$-parameter  automorphism group acting
    on \mA. Then $\Psi $  and the  generators commute, whenever 
    defined, i.e.,
    \begin{equation}
      \label{eq:9}
      \psi_s (\delta_t (a)) = \delta_t( \psi_s( a)) \text{ for } a \in \mD(\delta_t,\mA), s,t \in \bbRd .
    \end{equation}
  \item[(iii)] Derived spaces consist of continuous elements: $\mA^{(1)}=\bigcap_{k=1}^d \mD(\delta_k,\mA) \subseteq
    C(\mA)$.
  \item[(iv)] Let $\mD_{s,t}= \mD(\delta_s, C(\mA)) \cap \mD(\delta_t, C(\mA)) \cap \mD(\delta_s\delta_t,
    C(\mA))$. Then for $s,t \neq 0$
    \begin{equation*}
      \mD_{s,t}=\mD_{t,s}, \text{ and } \delta_s\delta_t=\delta_t\delta_s \text{ on } \mD_{s,t}.
    \end{equation*}
  \end{itemize}
\end{prop}

\begin{defn} \label{def:cnspace}
  For $k\in \bbN _0$  the spaces $C^k(\mA)$ and $C^\infty(\mA)$ are defined as
  \begin{equation*}
    C^k(\mA)=\bigcap_{\abs{\alpha} \leq k} \mD(\delta^\alpha,C(\mA)) \quad
    \text{and} \quad  C^\infty(\mA)=\bigcap_{\alpha\geq 0}
    \mD(\delta^\alpha,C(\mA)) \, .
  \end{equation*}
The   norm on $C^k(\mA)$ is $ \norm{a}_{C^k(\mA)}=\sum_{\abs{\alpha}\leq k}\frac{1}{\alpha!}\norm{\delta^\alpha a}_\mA.  $ For
  $k=0$ we set $C^0(\mA)=C(\mA)$.
\end{defn}
Proposition \ref{prop:density} shows that this definition does not depend on the ordering of the standard basis.

It is a (trivial but) important fact that the smoothness spaces consist of the \cont\ elements of the derived
spaces, i.e., 
\begin{equation}
  \label{eq:12}
  C(\mA^{(1)})=C^1(\mA).
\end{equation}
Algebra properties and \IC ness of the spaces $C^k(\mA)$ are summarized in the following proposition. Note that in contrast to Theorem \ref{thm:derivationsIC} we do not need any further assumptions on \mA.
\begin{prop} \label{prop:smooth_IC} Each $C^k(\mA)$ is an \IC\ Banach subalgebra of $\mA$. $C^\infty(\mA)$ is an
  \IC\ Fr\'{e}chet subalgebra of $\mA$. 
\end{prop}
\begin{proof}
  By Proposition~\ref{prop_comm-deriv} $C^k(\mA )$ is inverse-closed in $C(\mA )$ and $C(\mA ) $ is
  inverse-closed in $\mA $, whence $C^k(\mA ) $ is inverse-closed in $\mA $. If $a\in C^\infty (\mA )\subseteq
  C^k(\mA )$, $k\geq 0$, is invertible in $\mA $, then $a^{-1} \in C^k(\mA )$ for all $k\geq 0$ and thus $a^{-1}
  \in C^\infty (\mA )$.
\end{proof}

We summarize the inclusion relations between the derived spaces $\mA^{(k)}$ and the spaces $C^k(\mA)$.
\begin{equation}
  \label{eq:smoothincl}
  \mA \supseteq C(\mA) \supseteq \mA^{(1)} \supseteq C^1(\mA)=C(\mA^{(1)}) \supseteq \mA^{(2)} \supseteq \cdots \supseteq C^\infty (\mA)=\mA^{(\infty)}
\end{equation}
In general, $C(\mA^{(k)})$ is not dense in $\mA^{(k)}$, but $C^\infty(\mA)$ is dense in $C(\mA)$. The inclusions follow from Proposition~\ref{prop:density}(iii) and (\ref{eq:12}).

\subsection*{Smoothness in Matrix Algebras}
\label{sec:smoothn-matr-algebr}
We now identify the derived spaces $\mA^{(k)}$ and the spaces $C^k(\mA)$ for some of the \MA s of
Section~\ref{sec:standard-examples} with respect to the automorphism group $\{\chi _t\}
$. 
\begin{prop} \hspace{1cm} 
  \begin{enumerate}
  \item [(i)]Let $r\geq0$ and $\mA$ be one of the algebras $\mJ_r,\mS_r, \bopzd$. Then $C(\mA) \neq \mA$.
  \item [(ii)]$C^k(\mC_s)=\mC_{k+s}$,  $k\in \bbN_0, s\geq 0$.
  \item [(iii)]$A \in C(\mJ_r) \Leftrightarrow \lim_{k \to \infty} \norm{\hat A (k)}_{\mJ_r}= \lim _{k \to \infty}
  \|\hat A (k)\|_{\ell^2\to \ell^2} (1+|k|)^r = 0 $.
  \end{enumerate}
\end{prop}
\begin{proof}
  (i) Define the anti-diagonal matrix $\Gamma_r$ by $\Gamma_r(k,-k)=
  (1+\abs {2k})^{-r}, k\in \bbZd $ and $\Gamma _r (k,l) = 0$ for
  $l\neq -k$.  Then $\Gamma_r \in \mJ_r$ and  $\Gamma _r \in \mS_r$, and in fact $\|\Gamma _r\|_{\mJ_r} =
  \|\Gamma _r\|_{\mS_r} = 1$. Likewise, $\Gamma _0$ is unitary in $
  \bop$. The matrix 
$\chi_t(\Gamma_r)-\Gamma_r$ has the non-zero entries on the anti-diagonal
  \begin{equation*} \label{eq:41} 
(\chi_t(\Gamma_r)-\Gamma_r)(k,-k)=\abs{\cexp [(k+ k)\cdot t] -1}\,
\Gamma_r(k,-k)=2 \, \abs{\sin(2 \pi k \cdot
      t)}\, (1 +\abs{2k})^{-r}, \quad k\in \bbZd \, .
  \end{equation*}
  The norm in $\mJ_r$ and $\mS_r$ is thus
  \[
  \norm{\chi_t(\Gamma_r)-\Gamma_r}_{\mJ_r}=\norm{\chi_t(\Gamma_r)-\Gamma_r}_{\mS_r}=2 \sup_{k\in \bbZd} \abs{\sin(2 \pi k \cdot t)},
  \]
  and so $\limsup _{\abs t\to 0}
  \norm{\chi_t(\Gamma_r)-\Gamma_r}_{\mJ_r} = \limsup _{\abs t\to 0}
  \norm{\chi_t(\Gamma_r)-\Gamma_r}_{\mS_r}= 2$. Similarly, $\limsup 
  _{\abs t\to 0} \norm{\chi_t(\Gamma_0)-\Gamma_0}_{\ell^2 \to \ell^2}=2$.  So $\Gamma _r\not \in C(\mJ _r ) \cup C(\mS _r )$ and $\Gamma _0 \not \in
  C(\bop ) $.

  (ii) We first verify that $C(\mC_r)=\mC_r$ for all $r \geq 0$ by a
  direct calculation (or by applying
  Proposition~\ref{prop_DeLeeuw}). Consequently $C^k(\mC _r) = (\mC
  _r)^{(k)}$ according to Definition~\ref{def:cnspace}.  Now Proposition
  \ref{prop-solidmazd} and ~(\ref{eq:ch8}) imply that   $
  \mC_r^{(k)}=(\mC_r)_{v_k}=\mC_{r+k}$.

  (iii) 
  First let $A \in C(\mJ_r)$. Then for every $\epsilon>0$ there is a $\tau = \tau(\epsilon )$ such that
  \[
\|\chi_t(A)-A \|_{\mJ _r} = 2  \sup _{k\in \bbZd}  \abs{\sin \pi k \cdot t} \, \norm{\hat A (k)}_{\mJ_r} < \epsilon
  \]
 for all $t$ with $\abs t < \tau $. If $\abs k _2 > (2\tau )^{-1}$ and
   $t=\frac{k}{2\abs k_2^2}$, then  $ \norm{\hat A (k)}_{\mJ_r} <
   \epsilon$, and so $\lim _{k\to \infty } \|\hat 
  A(k)\|_{\mJ _r} = 0$.

  For the converse implication write
  \begin{equation*}
    \norm{\chi_t(A)-A}_{\mJ_r} \leq \max_{\abs k < N} \norm{\hat A
      (k)}_{\mJ_r} \abs{\cexp [k \cdot t] -1} + 2 \sup_{\abs k \geq N}\norm{
      \hat A (k)} _{\mJ_r}\, .
  \end{equation*}
  This expression can be made arbitrarily small by choosing $N$ sufficiently large first and then letting $t$ tend to
  zero. Consequently, $A\in C(\mJ _r)$.
\end{proof}

Without proof we mention that
a matrix $A$ is in $C(\mS_0)$ if and only if
\begin{equation}
  \lim_{N \to \infty} \sup_{k\in \bbZd}  \sum_{\abs s > N} \abs{A(k,k-s)} =0 \text{ and }
  \lim_{N \to \infty} \sup _{k\in \bbZd}  \sum_{\abs s > N} \abs{A(k-s,k)} =0.
\end{equation}
This can be shown by hand, but will follow immediately from Corollary~\ref{Cor_perWeierstrass}.

\subsection{\HZ\ Spaces and Generalized Smoothness}
\label{sec:lipsch-spac-gener}
In analogy with the H\"{o}lder-Zygmund spaces on \bbRd\ we now define the \HZ\ spaces related to the
\BA\ \mA. This concept is well known for semigroups acting on \BS s, see \cite[Ch.~3]{Butzer67},
\cite{Engel00}. 

We gather some notation. Let $\Psi$ be an  automorphism group  on
\mA. For  $t \in \bbRd$ the
finite differences of $a \in \mA$  are defined as
\[
\Delta_t a = \psi_t(a)-a, \quad \Delta^k_t a = \Delta_t \Delta^{k-1}_t a,\quad k\geq1.
\]
The $k$-th modulus of smoothness is given by
\[
\omega^{(k)}_h(a)= \sup_{\abs t \leq h}\norm{\Delta^k_t a}_\mA, \quad h>0.
\]
We set $\omega_h(a)=\omega^{(1)}_h(a)$.
For $0 < r \leq 1$ the \emph{\HZ\ seminorm} of $a\in \mA$ is
\begin{equation}
  \label{lip*a}
  \abs{a}_{\Lambda_r}= \sup_{\abs t \neq 0 } \abs{t}^{-r}{\norm{\Delta^2_t a}_\mA}.  
\end{equation}
It is easily seen to be equivalent to $ \sup_{\abs t \neq 0 }\, \abs{t}^{-r}\omega^{(2)}_{\abs t}(a)$.
\begin{defn} \label{def:lipsch-spac-gener} Given $0 \leq r < \infty$
  with $r=k+\eta$, $k\in \bbN _0$ and $ 0< \eta 
  \leq 1$, the H\"older-Zygmund space $\Lambda_r(\mA)$ consists of all $a \in \mA$ for which
  \begin{equation}
    \label{eq:Hoeldzygabst}
    \norm{a}_{\Lambda_r(\mA)}= \norm{a}_{C^k(\mA)}+ \sum _{|\alpha|= k}       \norm{\delta^\alpha (a)}_{\Lambda_{\eta}} < \infty \, .
  \end{equation}
  The subspace $\lambda_r(\mA)$ consists of all $a \in C^k(\mA)$, such that
  \begin{equation}
    \lim_{ t \to 0 } \abs{t}^{-\eta}{\norm{\Delta^2_t\delta^\alpha (a)
      }_\mA} =0   \qquad \text{ for all } \,  \alpha, \abs \alpha =k \, .
  \end{equation}

\end{defn}
\begin{rems}
  For $\mA = C(\bbRd)$ and the translation group $\Psi=\set{T_t}$ the spaces $\Lambda_r(C(\bbRd))$ coincide with
  the classical H\"older-Zygmund spaces.

  The ``small'' \HZ\ space $\lambda_r(\mA)$ can be identified with $C(\Lambda_r(\mA))$.

  There are many equivalent definitions of \HZ\ spaces on $\bbRd$.  These carry over to $\Lambda_r (\mA)$. 
We will need the following characterizations.
\end{rems}
\begin{lem} \label{lemma:weakdeff} \hspace{1cm}
\begin{enumerate}
\item[(i)] Weak definition: For $a \in C(\mA)$ and $a' \in \mA'$ (the dual of \mA) we define
  \begin{equation} \label{eq:45} G_{a',a} (t)=\inprod{a',\psi_t(a)},
  \end{equation}
where $\inprod{\,,\,}$ denotes the dual pairing of $\mA' \times \mA$.  Then for $r >0$
  \[
  \norm{a}_{\Lambda_r(\mA)} \asymp \sup_{\norm{a'}_{\mA'}\leq 1} \norm{G_{a',a}}_{\Lambda_r(\bbRd)} .
  \]
\item[(ii)] First order differences: If $0 < r <1$, then
 the expressions $\sup_{\abs t \neq 0
  }\abs{t}^{-r}{\norm{\Delta_t a }_\mA} $
and $\sup_{\abs t \neq 0}\abs{t}^{-r}\omega_{\abs t}(a)$ are
  equivalent seminorms on $\Lambda_r(\mA)$.
\item[(iii)] Higher order differences: Let $k \in \bbN$, $0 <r <k$. Then $\sup_{\abs t \neq
    0}\abs{t}^{-r}\norm{\Delta^k_t a}_\mA$ and $\sup_{\abs t \neq 0}\abs{t}^{-r}{\omega^{(k)}_{\abs t}( a)}_\mA$ are
  equivalent seminorms on $\Lambda_r(\mA)$.
\end{enumerate}
\end{lem}
\begin{proof}
    We prove (i) directly from Definition~\ref{def:lipsch-spac-gener}.
   Note first that
  \[
  \norm {a}_\mA \asymp \sup_{\norm{a'}_{\mA'}\leq 1}\norm{G_{a',a}}_\infty \, ,
  \]
  because $ \norm{G_{a',a}}_\infty \leq \|a'\|_{\mA '} \| \psi _t (a) \|_{\mA} \leq M_\Psi \|a\|_{\mA } \,
  \|a'\|_{\mA ' } $ and
$$
\norm{a}_{\mA } = \sup _{\norm{a'}_{\mA'}\leq 1} |\langle a', a\rangle | \leq \sup _{\norm{a'}_{\mA'}\leq 1, t\in
  \bbRd} |\langle a', \psi _t( a)\rangle | = \sup _{\norm{a'}_{\mA'}\leq 1} \norm{G_{a',a}}_\infty \, .
$$
We prove the equivalence of the $\Lambda_r$-seminorms for $r \leq 1$ first. Using the algebraic
identity
\[
\inprod{a',\Delta^2_s a}=\Delta^2_s \inprod{a',\psi_t (a)} \vert _{t=0}=\Delta^2_s G_{a',a} \vert_{t=0} \, ,
\]
we obtain
\begin{equation}
  \label{eq:ch5}
  \abs{a}_{\Lambda_r}=\sup_{s\neq 0} \,  \abs s ^{-r} \norm{\Delta^2_s
    a}_\mA \asymp \sup_{s \neq 0} \abs s ^{-r}\sup_{\norm{a'}\leq1}
  \norm{\Delta^2_s G_{a',a}}_\infty  = \sup_{\norm{a'}\leq1}
  \norm{G_{a',a}}_{\Lambda _r} \, .
\end{equation}
\\
For $ r > 1$ we make use of $ \inprod{a',\delta^\alpha (a)}=D^\alpha G_{a',a} \vert_{t=0}, $ and obtain 
\begin{equation}
  \label{eq:ch4}
  \norm{\delta^\alpha (a)}_\mA \asymp \sup_{\norm{a'}\leq1}
  \norm{D^\alpha G_{a',a}}_\infty \, .  
\end{equation}
Combining~\eqref{eq:ch4} and \eqref{eq:ch5}, we obtain  $\abs a_{\Lambda _r (\mA )}\asymp  \sup_{\norm{a'}\leq 1} \norm{G_{a',a}}_{\Lambda _ r (\bbRd)}$.

Assertions (ii) and (iii) follow from (i) and the  well-known  scalar case. 
\end{proof}

From the standard literature \cite[Ch. 3.1, 3.4]{Butzer67} we know that every $\Lambda _r (\mathcal{A})$,
$r>0$ is a Banach space.  Furthermore $\Lambda _ r (\mathcal{A})$ is invariant under the action of
$\Psi$ and the following continuous embedding holds for $r \leq s $.
\begin{equation} \label{eq:47} \Lambda_s (\mA) \subseteq \Lambda_r(\mA) \, .
\end{equation}
Our interest is in the algebra property and the inverse-closedness of $\Lambda _r (\mA )$.
\begin{thm}\label{thm_HZ_is_ICBA}
  Let $\mathcal{A}$ be a Banach algebra, $\Psi $ be a $d$-dimensional
 automorphism group acting on 
  $\mathcal{A}$  and $r >0$. Then $\Lambda_r (\mA)$ is a Banach
  subalgebra of $\mathcal{A}$  and $\Lambda_r 
  (\mathcal{A})$ is \IC\ in \mA.
\end{thm}

\begin{proof} 
  We first treat the case $r \leq 1$. Taking norms in the identity
  \begin{equation}\label{eq:deltsquare}
    \Delta^2_t(ab) 
    =\psi_{2t}(a)\Delta^2_t b + 2 \psi_t(\Delta_t a) \Delta_t
    b+(\Delta^2_ta) \, b \, ,
  \end{equation}

  we obtain
$$
\norm{\Delta^2_t (ab)}_\mA \leq M_\Psi (\norm{a}_\mA \norm{\Delta^2_tb}_\mA + 2 \norm{\Delta_t a}_\mA
\norm{\Delta_t b}_\mA + \norm{\Delta^2_t a}_\mA \norm{b}_\mA)\, .
  $$
  Consequently, using Lemma~\ref{lemma:weakdeff}(ii) 
  \begin{equation*}
    \abs{ab}_{\Lambda_r} = \sup _{t\neq 0} |t|^{-r }
    \norm{\Delta^2_t (ab)}_\mA  \leq C (\norm{a}_\mA
    \abs{b}_{\Lambda_r}+\abs{a}_{\Lambda_{r/2}}\abs{b}_{\Lambda_{r/2}}+
    \abs{a}_{\Lambda_r} \norm{b}_\mA) . 
  \end{equation*}
  To get rid of the $\Lambda _{r /2}$-norm, we use the embedding (\ref{eq:47}) $\abs{a}_{\Lambda_{r/2}} \leq C
  \norm{a}_{\Lambda_r}$, and we finally obtain
  \begin{equation*}
    \norm{ab}_{\Lambda_r}=\norm{ab}_\mA+\abs{ab}_{\Lambda_r} \leq C \norm{a}_{\Lambda_r}\norm{b}_{\Lambda_r},
  \end{equation*}
  which shows that $\Lambda _r (\mA )$ is a Banach algebra for $r \leq 1$.

   Next we verify the \IC ness of $\Lambda_r (\mA )$. Let $a \in \Lambda_r (\mA )$ and $a$ invertible in \mA. 
We use (\ref{eq:deltsquare}) with $b=a^{-1}$ and obtain
\begin{equation}\label{eq:5}
\Delta_t^2(a^{-1})=-\psi_{2t}(a^{-1})\bigl[2 \psi_t(\Delta_t(a))\Delta_t(a^{-1})+\Delta_t^2(a)a^{-1} \bigr] .
\end{equation}
Using
  \begin{equation}
    \label{eq:invderi}
    \Delta_t(a^{-1})=-a^{-1} \;  \Delta_t(a) \; \psi_t(a^{-1}),
  \end{equation}
we argue as above and arrive at
\[
\abs{a \inv}_{\Lambda_r} \leq C \norm{a^{-1}}^2_\mA \bigl(\abs{a}_{\Lambda_{r/2}}^2 \norm{a^{-1}}_\mA+ \abs{a}_{\Lambda_r} \bigr),
\]
which is finite, again by (\ref{eq:47}).

Now let us sketch the modifications required to treat the general case
$r=k+\eta$, $k \in \bbN$, $0<\eta \leq 1$. If $a,b\in \Lambda _r(\mA
)$, then $a,b \in C^k(\mA )$ and $\delta ^\alpha (a), \delta ^\alpha
(b) \in \Lambda _\eta (\mA )$ for $|\alpha |=k$. Since $\Lambda _\eta
(\mA )$ is a \BA\ by the preceding step, the general Leibniz rule
\eqref{leib} implies that $\delta ^\alpha (ab)$ is in $\Lambda _\eta
(\mA )$ for $|\alpha |=k$, whence $\Lambda _r(\mA )$ is a \BA . 

To show that  $\Lambda_r(\mA)$ is \IC\  in \mA ,  we assume that $a
\in \Lambda_r(\mA)$ and $a \inv \in \mA$. From 
Proposition~\ref{prop:smooth_IC} we  know already  that $a \inv \in 
C^k(\mA)$, i.e.,  $\delta^\alpha(a \inv) \in C(\mA)$ for $\abs \alpha \leq k$. 
Now, using \eqref{leib} with $b=a\inv $, we obtain an explicit
expression for $\delta ^\alpha (a\inv ), |\alpha|=k$, namely
\begin{equation}
  \label{eq:ch79}
  \delta ^\alpha (a\inv )= - \sum _{0\neq \beta \leq \alpha }
  \binom{\alpha}{\beta} \delta ^\beta (a) \delta ^{\alpha-\beta }
  (a\inv )
  \, .
 \end{equation}
By assumption $\delta ^\beta (a) \in \Lambda _\eta (\mA )$ for $\beta
\leq \alpha $  and $\delta
^{\alpha - \beta } (a\inv ) \in C^1(\mA )\subseteq \Lambda _\eta (\mA
)$ for $\beta \neq 0$.  Consequently all terms on the right-hand side
of \eqref{eq:ch79} are in $\Lambda _\eta (\mA)$ and therefore $\delta
^\alpha (a\inv )\in \Lambda _\eta (\mA )$ for $|\alpha |=k$. We have
proved that $a\inv \in \Lambda _r (\mA )$ and thus $\Lambda _r(\mA )
$ is \IC\ in $\mA $. 
\end{proof}

What does Theorem~\ref{thm_HZ_is_ICBA} say about concrete \MA s? In line with our general philosophy
we show next how the abstract smoothness is related to the
off-diagonal decay of  matrices. 
\begin{prop}\label{prop:solid_MA_HZ}
  Let \mA\ be a solid \MA\ over \bbZd\ and $r>0$. Then
  $\Lambda_r(\mA)$ is solid, and $\mA_{v_r} \subseteq 
  \Lambda_r(\mA)$.
\end{prop}
\begin{proof}
  Recall that the automorphism group is given by $\chi _t (A) = M_t A M_{-t}$ and $(\chi _t(A))(k,l) = e^{2\pi i
    (k-l)\cdot t} A(k,l)$.  For $0<r\leq 1$  the seminorm $\abs{A} _{\Lambda_r (\mathcal{A})}$ is the \mA-norm of the matrix
  with entries
  \begin{equation}\label{eq:48}
    \abs{t}^{-r} \, \abs{\chi_{2t}(A) -2\chi_t(A)+A}(k,l)=\abs{A(k,l)}
    \frac{\abs{\sin^2{\pi (k-l)\cdot t}}}{\abs t^r}.
  \end{equation}
  If $A\in \mA $ and $\abs{B(k,l)} \leq \abs{A(k,l})$, $k,l\in \bbZd$, then the solidity of $\mA $ implies not
  only that $\norm{ B}_{\mA} \leq \norm{A}_{\mA} $, but by \eqref{eq:48} also that
$$|B|_{\Lambda _r (\mA )} \leq |A|_{\Lambda _r (\mA )} \, ,$$
and thus $\Lambda _r (\mA ) $ is solid.

For $t \neq 0$ we obtain
\[
\abs{A(k,l)} \frac{\sin^2 (\pi (k-l)\cdot t )}{\abs t^r}= \abs{A(k,l)} \, \frac{\sin^2 (\pi (k-l)\cdot t )}
{({\abs{ \pi (k-l)}\abs t)}^r} \pi ^r \abs{k-l}^r \leq \pi ^r \abs{A(k,l)} \abs{k-l}^r \, .
\]
Applying the $\mA $-norm to both sides of this inequality, we see that $|A|_{\Lambda _r (\mA)} \leq \pi ^r
\|A\|_{\mA _{v_r}} \, $ and thus  $\mA_{v_r} \subseteq 
  \Lambda_r(\mA)$.

If $0< k < r \leq k+1$ for $k\in \bbN $, we apply the same argument to
all $\delta^\alpha(A)$, $\abs \alpha =k$. Details are left to the reader.
\end{proof}
It is possible but non-trivial (see~\cite{Klo09}) to show that for  a
solid \MA\  $\mA $
\[
\Lambda_r(\mA) \subseteq \mA_{v_s} \quad \text{for all } s < r.
\]

For the Jaffard class we obtain a complete characterization of the \HZ\ spaces.
\begin{prop}\label{prop:Jaffard-char}
  Let $r, s>0$. Then
  \begin{equation}
    \label{eq:jaffHZ}
    \Lambda_r(\mJ_s)=\mJ_{s+r}.
  \end{equation}  
\end{prop}
\begin{proof}
  By ~\eqref{eq:ch8} and
  Proposition~\ref{prop:solid_MA_HZ},  $\mJ_{s+r} = (\mJ _{s})_{v_r }
  \subseteq \Lambda_r(\mJ_s)$. 

For the converse assume first that $0<r\leq 1$ and  use (\ref{eq:48}) to obtain
  \[
  \norm{A}_{\Lambda_r(\mJ_s)}=\sup_{t \neq 0} \sup_{k\in \bbZd}
  \norm{\hat A (k)}_{\ell^2 \to \ell^2 } (1+\abs k)^s \, \frac{\sin^2 
    (\pi k\cdot t)} {{\abs t}^r} \, .
  \]
 So $A \in \Lambda_r(\mJ_s)$ implies
  \[
  \norm{\hat A (k)}_{\mJ_0} (1+\abs k)^s {\abs{\sin^2{(\pi k\cdot t)}}} \leq C \abs{t}^r
  \]
  for all $k \in \bbZd$ and $ t \in \bbRd, t\neq 0$. If $t=\frac{k}{2 \abs k_2^2}$ we conclude that $\norm{\hat A
    (k)}_{\ell^2 \to \ell^2 } \leq C \abs{k}^{-r-s }_2$, that is $A \in \mJ_{s+r}$.

If $A\in \Lambda _{r}(\mJ _s)$ for $r=k+\eta >1$, $k \in \bbN$, $0<
\eta \leq 1$, then by definition $\delta ^\alpha (A) \in \Lambda
_\eta (\mJ _s) = \mJ _{s+\eta }$ for each $\alpha$ with $\abs \alpha
=k$. This means that $A$ belongs to the derived algebra $(\mJ _{s+\eta
})^{(k)}$. Since $(\mJ _{s+\eta
})^{(k)} = \mJ _{s+\eta +k}$ by  Proposition~\ref{prop-solidmazd}, we  obtain that $A \in
\mJ_{s+\eta+k}$. We have proved that $\Lambda _{r}(\mJ _s)
\subseteq \mJ _{r+s}$ for all parameters $r,s>0$.  
\end{proof}

A more elementary relation between H\"older-Zygmund class and \odd\ is valid in all \MA s.
\begin{prop}
  Let \mA\ be a \MA. If $A \in \Lambda_r(\mA)$, then $\norm{\hat{A}(k)}_\mA =\bigo(\abs{k}^{-r})$.
\end{prop}
\begin{proof}
We remark first that the  $k$-th side diagonal of $A\in C(\mA )$  is
exactly  the $k$-th   ``Fourier coefficient'' of the mapping $t\to
\chi _t(A)$:  
  \begin{equation}
    \label{eq:10}
    \hat A (k) = \int_{\bbT^d} \chi_t(A) e^ {- 2 \pi i k t} \; dt\, .
  \end{equation}
This  can be seen by direct calculation or by using~\cite{Bas90}. 
Then the  standard
  argument for the decay of the Fourier coefficients of $f\in \Lambda
  _r(\bbT ^d)$  \cite[Theorem I.4.6]{Katznelson} carries over to
  $\Lambda _r(\mA  )$.
\end{proof}

\section{Approximation in Banach Algebras}
\label{sec:appr-banach-algebr}
In this section we study a completely different method for the construction of \IC\ subalgebras. We assume the
existence of a nested set of subspaces and study the corresponding
approximation spaces (see, e.g.~\cite{Butzer68,
  DeVore93,Pietsch81}). The analogy is now with the approximation of
periodic functions by trigonometric polynomials.

Let  the index set $\Lambda$  be either $\bbR^+_0$ or $\bbN_0$. An \emph{approximation scheme} on the \BA\ \mA\
is a family $(X_\sigma)_{\sigma \in \Lambda}$ of closed subspaces $X_\sigma$ that fulfill the conditions
\begin{equation}
  \label{eq:2}
  X_0=\set{0}\quad \text{ and } \,\, \,  X_\sigma \subseteq X_\tau\text{ for } \sigma \leq \tau, \text{ and}
\end{equation}
\begin{equation}
  \label{eq:6}
  X_\sigma \cdot X_\tau\subseteq X_{\sigma+\tau}, \quad \sigma,\tau\in \Lambda.
\end{equation}
If \mA\ possesses an involution, we further assume that
\begin{equation}
  \label{eq:26}
  \one \in X_1 \quad  \text{ and } \quad  X_\sigma=X^*_\sigma \quad \text{for all } \sigma \in \Lambda.
\end{equation}
The \emph{$\sigma$-th approximation error} of  $a \in \mA$ by $X_\sigma$ is
\begin{equation}
  \label{eq:apperr}
  E_\sigma(a)=\inf_{x \in X_\sigma} \norm{a-x}_\mA. 
\end{equation}
\begin{prop} \label{prop:apprelements1} Let \mA\ be a \BA\ with an approximation scheme $(X_\sigma)$. The set
  \begin{equation}
    \label{eq:37}
    \mA_0 =\set{a \in \mA: \lim_{\sigma \to \infty} E_\sigma(a)= 0}=\overline{\bigcup_{\sigma \in \Lambda} X_\sigma}^\mA    
  \end{equation}
  is a closed subalgebra of \mA. If \mA\ is symmetric, then $\mA_0$ is \IC\ in \mA.
\end{prop}
\begin{proof}
  Identity (\ref{eq:37}) is straightforward. With (\ref{eq:6}) we obtain that $\mA_0$ is a \BA. Furthermore,
  since $\mA_0$ is a closed $*$-subalgebra of the symmetric algebra
  \mA, $\mA _0$ is \IC\ in $\mA $ (see the remark after
  Proposition~\ref{prop:lemma-hulanicki}). 
\end{proof}
By specifying a rate of decay for $E_\sigma(a)$ as $\sigma \to \infty$, we may define a class of approximation
spaces in \mA\ by the norm 
\begin{equation}
  \label{eq:appspace}
  \norm{a}_{\mE_r^p}^p= \begin{cases}
    \int_0^\infty{E_\sigma(a)^p}(\sigma+1)^{r p}\frac{
      \dd \sigma}{\sigma+1},& \text{ for }  \Lambda=\bbR^+,\\
    \sum_{k=0}^\infty{E_k(a)^p}(k+1)^{r p}
    \frac{1}{k+1},&  \text{ for } \Lambda=\bbN_0 \, ,
  \end{cases}
\end{equation}
for $1\leq p < \infty $ with the  obvious change for $p=\infty$. 
The elementary properties of $\mE_r^p(\mA) $ were already obtained in \cite{DeVore93,Pietsch81}, and in
\cite{Almira01,Almira06}.

\begin{prop}[~\cite{Almira01,Almira06}] \label{appspBA} Let \mA\ be a \BA\ an approximation scheme
  $(X_\sigma)_{\sigma \in \Lambda}$.  Then $\mE_r^p(\mA)$ is a \BA\ and dense in $\mA_0$ for every for $1\leq
  p\leq \infty$ and $r>0$.
\end{prop}
\begin{proof}
  We give the proof only for the index set $\Lambda=\bbN_0$. 
  Choose $a_n, b_n \in X_n$ such that $\norm{a-a_n}\leq 2 E_n(a)$ and $\norm{b-b_n} \leq 2E_n(b)\leq 2\|b\|_{\mA } $. Then
  $\norm{b_n } \leq \|b\|+ \|b_n -b\| \leq 3\|b\|$ and
  \begin{equation}\label{eq:frue}
    \begin{split}
      E_{2n+1 }(ab) & \leq E_{2n}(ab) \leq
      \norm{ab-a_n b_n}_\mA \\
      &\leq \norm{a}_\mA\, \norm{b-b_n}_\mA +\norm{b_n}_\mA
      \norm{a-a_n}_\mA \\
      &\leq 2\norm{a}_\mA E_n(b)+ 6 \norm{b}_\mA E_n(a).
    \end{split}
  \end{equation}
  Using this estimate and the equivalence $(1+n ) \asymp (1+2n )$, we 
  obtain
  \begin{equation} \label{twonormrel} \norm{a b}_{\mE^r_p} \le C \left(\norm{a}_{\mE_r^p} \norm{b}_\mA
      +\norm{b}_{\mE_r^p} \norm{a}_\mA\right).
  \end{equation}
 The \BA-property of $\mE_r^p(\mA)$ now follows from \eqref{twonormrel}.  The claimed density follows from the
  definition of the approximation spaces.
\end{proof}
We now treat the \IC ness of approximation spaces.
\begin{prop} \label{appspIC} Let \mA\ be a symmetric \BA\ and $(X_\sigma)_{ \sigma \in \Lambda}$ an
  approximation scheme. Then $\mE_r^p(\mA)$ is \IC\ in $\mA$.
\end{prop}
\begin{proof}
  The norm inequality \eqref{twonormrel} is exactly the hypothesis for the application of Brandenburg's trick
  (Section \ref{sec:method-brandenburg}),  so (\ref{twonormrel}) implies that
  \[
  \rho_{\mE_r^p}(a)=\rho_\mA(a), \quad \text{ for all }  \, a \in \mE_r^p(\mA).
  \]
  Since \mA\ is symmetric, Lemma \ref{prop:lemma-hulanicki} shows that
  $\mE_r^p(\mA ) $ is \IC\ in \mA.
\end{proof}
Proposition \ref{appspIC} is not entirely new. If $\mA _0 = \mA$, then it follows from a result of Kissin and
Shulman \cite[Thm.~5]{kissin94}. However, in most of our examples $\mA _0 \neq \mA $ and we only know that
$\mE_r^p(\mA ) $ is \IC\ in $\mA_0$, but nothing about $\mA $.  This is why the symmetry assumption is needed for the
proof of the \IC ness of $\mA_0$ in \mA. Our new  proof has the advantage of being short and concise.

We illustrate the preceding concepts with some examples.
\\ \\
(1) \emph{Approximation with trigonometric polynomials.} Let $\mA=L^\infty(\bbTd)$ and choose the approximation
scheme as
\begin{equation*}
  X_0=\set{0},\quad X_k=\spann \set{\cexp [r\cdot  t] \colon \abs{r} <k}, \: k\geq 1.
\end{equation*}
Clearly the  conditions (\ref{eq:2}-\ref{eq:26}) are  fulfilled and 
$\mA_0=C(\bbTd)$. 
 Proposition~\ref{appspIC} implies  that  $\mE^p_r(L^\infty(\bbTd))$
 is  \IC\ in $L^\infty(\bbTd)$.
\\ \\
(2) \emph{Approximation with banded matrices.} Let \mA\ be a \MA\ and 
let $\mT_N=\mT_N(\mA)$ be the set of matrices in \mA\ with bandwidth smaller than $N$,
\[
\mT_N=\set{A \in \mA \colon A=\sum_{\abs k < N}\hat A (k)}
\]
Then the sequence $(\mT_k)_{k\geq0}$ is an approximation scheme for \mA.
The closure of all banded matrices in \mA\ is the space of \emph{band-dominated matrices} in \mA
~\cite{Rabinovich98, rrs04}.
\begin{cor} \label{cor:apprmas} Let \mA\ be a symmetric \MA.
  \begin{enumerate}
  \item [(i)] Then the band-dominated matrices in \mA\ form a closed and \IC\ $*$-subalgebra of \mA.
  \item [(ii)] Each approximation space $\mE^p_r(\mA)$ is \IC\ in \mA.
  \end{enumerate}
\end{cor}
Theorem~\ref{thm-intro-matic} from the introduction follows immediately by choosing $p=\infty$ and $\mA=\bopzd$.

For the algebra of bounded operators on vector-valued $\ell^p$-spaces
special instances of (i) have been 
obtained in~\cite{Rabinovich98, rrs04}.

Loosely speaking, if a matrix can be well approximated by banded matrices, then its inverse can be approximated
by banded matrices with the same quality. 
This property expresses a form of off-diagonal decay, which we now
relate to the standard notions.
\begin{cor} \label{jaffident} \hspace{1cm}
\begin{enumerate}
\item[(i)] Assume that $\mA $ is a solid matrix algebra continuously
  embedded in $\mB (\ell ^2(\bbZd))$. Then  $\mE^\infty_r(\mA) \subseteq \mJ _r $, and $A\in \mE^\infty_r(\mA)$
  decays at least polynomially off the diagonal.
\item[(ii)] For the Jaffard algebra $\mJ_s$ we have
  \[
  \mE^\infty_r(\mJ_s)= \mJ_{s+r}.
  \]
  As a consequence of Corollary~\ref{cor:apprmas} $\mJ_{s+r}$ is \IC\ in $\mJ_s$ and in $\bopzd$ for $s > d$ and
  $r>0$. 
\end{enumerate}
\end{cor}

\begin{proof}
  (i) If $\mA$ is a solid \MA , then for $A \in \mA$ the banded matrix
  $\sum_{\abs k <n}\hat A (k)$ is a best 
  approximation to $A$ in $\mT_n$. Hence
  \begin{equation*}
    E_n(A)=\norm{A-\sum_{\abs k < n}\hat A(k)}_\mA =\norm{\sum_{\abs k
        \geq n}\hat A(k)}_\mA \, . 
  \end{equation*}
If $A\in \mE^\infty _r(\mA )$, the  size of the $n$-th diagonal is
majorized by 
  \[
  \|\hat A (n) \|_{\mA } \leq \norm{\sum_{\abs k \geq n}\hat A(k)}_\mA \leq \norm{A}_{\mE^\infty_r(\mA)}
  (n+1)^{-r}\, .
  \]
  Since $\mA $ is embedded into $\bop $, this implies that
$$
\|\hat A (n) \|_{\ell^2 \to \ell^2 } \leq \|\hat A (n) \|_{\mA } \leq \norm{A}_{\mE^\infty_r(\mA)} (n+1)^{-r}\, ,
$$
and thus $A \in \mJ _r$.

(ii) For $A\in \mJ_s$ we obtain
\begin{equation*}
  \label{eq:18}
  E_n(A)= \norm{\sum_{\abs k \geq n}\hat A(k)}_{\mJ_s} = \sup_{\abs{u-v}
    \geq n} \abs{A (u,v)}(1+\abs{u-v})^s. 
\end{equation*}
This means
\begin{equation*}
  \label{eq:19}
  \begin{split}
    A \in \mE^\infty_r(\mJ_s) &\Leftrightarrow E_n(A)(1+n)^r \leq C \text{ for all } n>0 \\
    & \Leftrightarrow \sup_{\abs{u-v} \geq n} \abs{A (u,v)}(1+\abs{u-v})^s (1+n)^r \leq C \text{ for all } n>0.
  \end{split}
\end{equation*}
This is true if and only if
\[
\|\hat A (n) \|_{\ell^2 \to \ell^2} (1+n)^{s+r } = \sup_{\abs{u-v} = n} \abs{A (u,v)}(1+\abs{u-v})^{s+r} \leq C \text{ for all }
n>0,
\]
and we have shown that $\|A\|_{\mJ _{s+r } } = \sup _{n\in \bbZd} \|\hat A (n)\|_{\ell^2 \to \ell^2} (1+|n|)^{s+r} <\infty $ 
 or $A \in \mJ_{s+r}$.
\end{proof}
This corollary helps to simplify the proof of Jaffard's orginal theorem in \cite{Jaffard90}. Suppose we already
know that $\mJ _{d+\epsilon }$ is inverse-closed in $\bop$ for $0<\epsilon \leq \epsilon _0$. By
Corollary~\ref{cor:apprmas} and~\ref{jaffident} $\mJ _{s}$, $s>d+\epsilon $, is inverse-closed in $\mJ
_{d+\epsilon }$ and hence in $\bop$. Thus it suffices to prove Jaffard's result for the range $d<r<d+\epsilon_0 $
for some small $\epsilon_0 >0$.
\\
\\
(3) \emph{Approximation in UHF algebras.} In order to illustrate the potential of approximation methods for
operator algebras, we discuss the approximation properties in \emph{UHF algebras}.

A \emph{uniformly hyperfinite (UHF) algebra} (Glimm ~\cite{Glimm60}) is the direct limit of a directed system
$\set{M_{n_k}, \phi_k}$ of full matrix algebras $M_{n_k}$. Precisely $M_{n_k}$ is the full algebra of $n_k \times
n_k$, $\{n_k\}$ is a sequence of positive integers $n_k$, such that $n_k$ divides $n_{k+1}$ ($n_{k+1}=r_k n_k$)
for all $k \in \bbN$ and $\lim _{k \to \infty } n_k = \infty $, and the unital embedding $\phi_k: M_{n_k} \to
M_{n_{k+1}}$ is given by $ A \mapsto A \otimes I_{r_k}$.
Suppressing the embedding maps, we can write
\begin{equation*}
  \mathrm{UHF}((n_k))=\mathrm{UHF}(\vec n)= \overline{\bigcup_k M_{n_k}}^{\bop}
\end{equation*}
and obtain a $C^*$-algebra. We refer to \cite{ Bratteli72, Glimm60} for the deeper properties of UHF algebras. Elements of
$\mathrm{UHF}(\vec n)$ can be understood as follows: Let
\begin{equation*}
  \phi_{k,\infty} \colon M_{n_k} \to \bopn; \quad A \mapsto 
  \begin{pmatrix}
    A & {} & {} \\
    {} & A & {} \\
    {} & {} & \ddots
  \end{pmatrix}
\end{equation*}
the natural embedding of $M_{n_k}$ into \bopn. Then any element of $\mathrm{UHF}(\vec n)$ can be written as a
limit in the operator norm on $\ell ^2(\bbN)$.
\begin{equation} \label{eq:24} A=\sum_k \phi_{k,\infty}(A_k), \qquad A_k \in M_{n_k}.
\end{equation}
\\
The very definition of the UHF algebras suggests a natural approximation scheme, namely the subalgebras
$M_{n_k}$. More precisely, let
\begin{equation*}
  X_0={0}, \quad X_k=\phi_{k,\infty}(M_{n_k}),\: k\geq 1.
\end{equation*}
In this situation, \eqref{eq:6} can be improved to
\begin{equation}
  \label{eq:21}
  X_n X_m \subseteq X_{\max (n,m)}.
\end{equation}
Property~\eqref{eq:21} implies an approximation result that is stronger than Propositions~\ref{appspBA} and \ref{appspIC}. In fact, choose an arbitrary weight function $w>0$ on $\bbN_0$ and define the
\emph{generalized approximation space} $\tilde E^p_w$, $1\leq p\leq\infty$, by the norm
\begin{equation*}
  \norm a_{\tilde E^p_w}=\left(\sum_{n\geq 0} E_n (a)^p w(n)^p\right)^{1/p}.
\end{equation*}
Since $E_0(a)=\norm a _\mA$ and $\norm{a}_\mA \leq \frac{1}{w(0)} \norm{a}_{\tilde E^p_w}$ the generalized
approximation space $\tilde E^p_w$ is embedded into \mA.  Since every
$X_k$ is an algebra, the estimate~(\ref{eq:frue}) can be improved to
\begin{equation}
  \label{eq:22}
  E_n(ab)\leq C\bigl(\norm{a}_\mA E_n(b)+\norm{b}_\mA E_n(a)\bigr) \, ,
\end{equation}
and consequently
\begin{equation}
  \label{eq:23}
  \norm{ab}_{\tilde E^p_w} \leq C\bigl(\norm{a}_\mA \norm{b}_{\tilde
    E^p_w}+ \norm{a}_{\tilde E^p_w} \norm{b}_\mA \bigr), 
\end{equation}
where $\mA=\mathrm{UHF}(\vec n)$.  Applying now Brandenburg's trick from Section~\ref{sec:method-brandenburg}, we obtain a
new class of ``smooth'' inverse-closed subalgebras of
$\mathrm{UHF}(\vec n)$. 
\begin{cor}
  For $1 \leq p \leq \infty$ and arbitrary $w>0$ the generalized
  approximation space $\tilde E^p_w$ is a dense 
  $*$-subalgebra of $\mathrm{UHF}(\vec n)$,  and  $\tilde E^p_w$ is
  \IC\ in $\mathrm{UHF}(\vec n)$. 
\end{cor}

\section{Smoothness and Approximation with Bandlimited Elements}
\label{sec:smoothn-appr-with}
So far the two constructions of \IC\ subalgebras are based on different structural features of \BA s, namely,
derivations or commutative automorphism groups, and approximation schemes. Again classical approximation theory
teaches us how to relate smoothness properties to approximation properties. The prototypes of such a connection
are the Jackson-Bernstein theorems for polynomial approximation of periodic
functions. 

In this section we develop a similar theory for \BA s with an automorphism group $\Psi $. The application to
matrix algebras then supports once more the insight that ``smoothness of matrices'' amounts to their
off-diagonal decay. 
The general setup is again that of Section~\ref{sec:smooth}. Let \mA\ be a unital \BA\ with a uniformly bounded
$d$-parameter group of automorphisms $\Psi$.  Let $\delta_{e_k}, 1\leq k \leq d$ denote the canonical generators
of $\Psi$. If \mA\ possesses an involution, we assume that $\Psi$ consists of $*$-automorphisms.

\subsection{Bandlimited Elements and Their Spectral Characterization}
\label{sec:spectr-subsp-band}
We first need an analogue of the trigonometric polynomials in the context of a \BA\ with an automorphism
group. 
\begin{defn}
  An element $a \in \mA$ is \emph{$\sigma$-bandlimited} for $\sigma >0$, if there is a constant $C$ such that
  \begin{equation}
    \label{eq:20}
    \norm{\delta^\alpha( a)}_\mA \leq C ( 2 \pi \sigma)^{\abs \alpha}
  \end{equation}
  for every multi-index $\alpha$. An element is \emph{bandlimited}, if it is $\sigma$-bandlimited for
  some $\sigma
  >0$. 
  Inequality \eqref{eq:20} is a generalized Bernstein inequality.
\end{defn}
\begin{ex}
  In $C(\bbT)$ the $N$-bandlimited elements are exactly the
  trigonometric polynomials of degree $N\in \bbN _0$. If $f$ is a
  trigonometric polynomial of degree $N$, then, by the classical
  Bernstein inequality, we have  $ \norm{f'}_\infty\leq 2
  \pi N \norm{f}_\infty.  $ This implies (\ref{eq:20}). 

 Conversely,
  if $f \in C(\bbT )$ is $N$-bandlimited in the 
  sense of \eqref{eq:20}, then
  \begin{equation*}
    C (2\pi N)^k \geq   \norm{D^kf}_{L^\infty(\bbT)}\geq
    \norm{D^kf}_{L^2(\bbT)}=\norm{((2 \pi i l)^k \hat f (l))_{l \in \bbZ}
    }_{\ell^2} \geq (2 \pi \abs m)^k \abs{\hat f(m)} \, 
  \end{equation*}
for all $m\in \bbZ $.  This is true for all $k\geq 0$, whence
$\hat{f}(m) = 0$ for $|m| >N$. 
  See~\cite[3.4.2]{Trigub04} for related statements.
\end{ex}
We next generalize Fourier arguments to obtain an alternative
characterization of \BL\ elements in a \BA. 
 To avoid vector-valued distributions, we need some technical preparation.
\begin{defn}
  The \emph{spectrum} of $a \in C(\mA)$ is
  \begin{equation}
    \label{eq:27}
    \spec(a)= \bigcup_{a' \in \mA'}\supp \,\mF(G_{a',a}),
  \end{equation}
  where the Fourier transform $\mF$ is used in the distributional sense and $G_{a'a,}(t)= \langle a', \psi _t a\rangle$
  was defined in (\ref{eq:45}).
\end{defn}
An equivalent but less convenient definition is given in
\cite[Def. 2.2.5]{bratteli96}.  

Here is a spectral characterization of $\sigma$-\BL\ elements in $\mA $.
\begin{prop}
  An element $a \in C(\mA)$ is $\sigma$-\BL, if and only if $\spec(a) \subseteq [-\sigma,\sigma]^d$.
\end{prop}
\begin{proof}
  Assume first that $\spec(a) \subseteq [-\sigma,\sigma]^d$. Then by definition
  \begin{equation*}
    \supp \mF(G_{a',a}) \subseteq [-\sigma,\sigma]^d \quad \text{ for all } a' \in \mA'.
  \end{equation*}
  By the Paley-Wiener-Schwartz theorem~\cite[3.4.9]{Trigub04} the bandlimited function $G_{a',a}$ can be extended to
  an entire function of exponential type $\sigma$, i.e., for every $\epsilon>0$ there is a constant $A =
  A(\epsilon) $, such that
  \[
  \abs {G_{a',a} (t+iy)} \leq A e^{(\sigma+\epsilon) \abs{y}} \qquad \text{ for } t,y \in \bbRd \, .
  \]
  Since $G_{a',a}= \langle a', \psi _t(a)\rangle $ is holomorphic for all $a' \in \mA '$, the mapping $t \mapsto
  \psi_t(a)$ is holomorphic.  This implies the existence of $\delta^\alpha(a)\in \mA$ for
  each multi-index $\alpha$.  To deduce \eqref{eq:20} we use the Bernstein inequality for entire
  functions~\cite[3.4.8]{Trigub04}, 
  \begin{equation}
    \label{eq:11}
    \norm{D^\alpha G_{a',a}}_\infty \le (2 \pi \sigma)^{\abs{\alpha}} \norm{G_{a',a}}_\infty
  \end{equation}
  for all $a' \in \mA'$. In particular, with~\eqref{eq:ch4}   
  \begin{equation}
\label{xxx}
    \begin{split}
      \|\delta ^\alpha (a)\|_{\mA } &\asymp \sup _{\|a'\|_{\mA '} \leq
        1} \norm{D^\alpha G_{a',a}}_\infty \\ 
      &\le (2\pi \sigma)^{\abs{\alpha}}
      \sup_{\norm{a'}_{\mA'}\leq 1}\norm{G_{a'a,}}_\infty \leq M_\Psi (2 \pi \sigma)^{\abs{\alpha}} \norm{a}_\mA.
    \end{split}
  \end{equation}
  Therefore $a$ is $\sigma$-\BL .

  Conversely, assume that $a$ is \BL\ with bandwidth $\sigma$. Then
  for arbitrary $t_0 \in \bbRd$ and $a'\in \mA  '$ the norm
  equivalence \eqref{eq:ch4} implies   that 
  \begin{equation}
    \label{eq:hmhm1}
    |D^\alpha G_{a',a}(t_0)| \leq  \|a'\|_{\mA '} \, \|\delta ^\alpha
    (a)\|_{\mA} \leq C (2\pi \sigma )^{|\alpha|}  \, .
  \end{equation}
Consequently the Taylor series of $G_{a',a}$ at $t_0$  converges
uniformly on $\bbRd$  and can be extended to an entire function
$$
G_{a',a}(z) = \sum _{\alpha \geq 0} \frac{D^\alpha
  G_{a',a}(t_0)}{\alpha !} (z-t_0)^\alpha \qquad \text{ for } z \in \bbCd \,
.
$$
The extension of $G_{a',a}$ is clearly independent of the base point
$t_0$ and satisfies the growth estimate
$$
|G_{a'a}(z)| \leq C \sum _{\alpha \geq 0} \frac{(2\pi \sigma
  )^{|\alpha|}}{\alpha !} |z-t_0|^{|\alpha|} \leq C e^{2\pi \sigma
  |z-t_0|} \, .
$$
For $z=t_0 +iy$, $y\in \bbRd$, we obtain  $|G_{a'a}(t_0+iy)| \leq C  e^{2\pi \sigma
  |y|}$, and   thus $G_{a',a}$ is an entire functions of exponential type
$\sigma$ \cite[4.8.3]{Timan63} for every $a'\in \mA '$. Once  
  again, the  Paley-Wiener-Schwartz theorem implies that  $\supp
  \mF(G_{a',a})\subseteq [-\sigma,\sigma]^d$ for all $a' \in \mA$.  We
  conclude   that $\spec(a)$ is contained in $[-\sigma,\sigma]^d$.
\end{proof}
\subsection{Periodic Group Actions}
\label{sec:peri-group-acti}
If the automorphism group $\Psi $ on \mA\ is \emph{periodic} (that is, there is a period $P \in \bbRd_+$ such
that $\psi_{t+P} = \psi_t$ for all $t$), the bandlimited elements can be described more explicitly by means of a
\BA-valued Fourier series. Without loss of generality we will assume that
$P=(1,\dotsc,1)$. 
Then the $k$-th Fourier coefficient of $a \in C(\mA$) is
\begin{equation}
  \label{eq:15}
  \hat a (k) = \int_{\bbTd} {\psi_t(a) \, e^{- 2 \pi i k \cdot t} }\, dt.
\end{equation}
By an observation of Baskakov~\cite{Bas90} for the action $\chi _t(A) = M_tAM_{-t}$ on a matrix $A$, the Fourier coefficient $\int _{\bbTd} \chi _t(A)e^{-2\pi i k\cdot t} \, dt $ 
is exactly the $k$-th side-diagonal $\hat A (k)$ of $A$ (see also (\ref{eq:10})). So there is no ambiguity in our notation.  The formal series $\sum_{k \in \bbZd} \hat a(k)
e^{2 \pi i k \cdot t} $ is the Fourier series of $a$ (see deLeeuw's work~\cite{deleeuw73,Deleeuw75, Deleeuw77}
for first developments of operator-valued Fourier series.)

\begin{prop}[ {\cite[Prop.~3.4]{Deleeuw75}}]\label{prop_DeLeeuw} Let
  $\Psi$ be a periodic  automorphism group  on \mA.  The following statements
  are equivalent for $a \in \mA$. 
  \begin{enumerate}
  \item[(i)] \label{item:1} $a \in C(\mA)$.
  \item[(ii)] \label{item:2} The Fejer-means of the Fourier series of
    $a$ converge in norm: 
    \[
    \psi_t (a)=\lim_{n \to \infty} \sum_{\abs k _\infty \leq n} \prod_{j=1}^d\Bigl(1-\frac{\abs {k_j}}{n+1} \Bigr) \hat a(k) e^{2 \pi i k \cdot t}.
    \]
  \item[(iii)] \label{item:3} The $C1$-means of the Fourier
    coefficients converge in norm to $a$: 
    \[
    a=\lim_{n \to \infty} \sum_{\abs k _\infty \leq n} \prod_{j=1}^d \Bigl(1-\frac{\abs {k_j}}{n+1} \Bigr) \hat a(k).
    \]
  \end{enumerate}
\end{prop}
DeLeeuw considers only the algebra \bop, but the proof for general \mA\ is identical. See also
\cite[2.12]{Katznelson}. 
An immediate consequence of Proposition~\ref{prop_DeLeeuw} is a Weierstrass-type density theorem for periodic
group actions.
\begin{cor}\label{Cor_perWeierstrass} \hspace{1cm}
  \begin{enumerate}
  \item [(i)] The set of \BL\ elements is dense in $C(\mA)$.
  \item [(ii)] $C^k(\mA )$ is dense in $C(\mA )$.
  \item [(iii)] An element $a \in \mA$ is $\sigma$-bandlimited, if and only if $\psi_t(a)$ is the trigonometric
    polynomial of the form
    \begin{equation}
      \label{eq:17}
      \psi_t( a ) = \sum_{\abs{k}_\infty \le \sigma}\hat a(k) e^{2 \pi i k\cdot t}
    \end{equation}
  \end{enumerate}
\end{cor}
We single out a characterization of \BL\ elements of \MA s.
\begin{cor}
  A matrix $A$ is banded with bandwidth $N$ in the \MA\ \mA, if and only it is $N$-bandlimited with
  respect to the group action $\{\chi _t\}$.
\end{cor}
 
\subsection{Characterization of Smoothness by Approximation}
\label{sec:char-smoothn-appr}
When working with an automorphism group on $\mA $, then the sequence of subspaces of bandlimited elements of
given bandwidth provides a natural approximation scheme for $\mA $. For this case, we will show that the
smoothness spaces defined in Section~\ref{sec:smooth} are equivalent to approximation spaces. In other words,
we will state and prove a general version of the Jackson-Bernstein theorem. 
Although proofs are similar to the classical ones in \cite{DeVore93,Timan63,Trigub04}, we gain new insight from
the generalization to \BA s. In particular, we need 
a theory of smoothness based on the action of an automorphism group (Section~\ref{sec:smooth}), and a spectral
characterization of \BL\ elements (Section~\ref{sec:spectr-subsp-band}).
Related results were obtained independently in~\cite{grushka07,Torba08}.
  \begin{lem}
    Let $\mA$ be a \BA\ with automorphism group $\Psi$, and set
    \begin{equation}
      \label{eq:29}
      X_0=\set{0}, \quad X_\sigma=\set{a \in \mA \colon \spec(a)
        \subseteq [-\sigma,\sigma]^d}, \quad \sigma>0. 
    \end{equation}
    Then $\{X_\sigma : \sigma \geq 0\}$ is an approximation scheme for $\mA $ consisting of the \BL\ elements.
  \end{lem}

  \begin{proof}
    If the group action is periodic, then
    Corollary~\ref{Cor_perWeierstrass}(iii) implies directly that
    $X_\sigma X_\tau     \subseteq X_{\sigma +\tau}$.
    If the acting group is $\bbRd$, then
we take norms in  the  Leibniz rule~\eqref{leib}  and  substitute the
estimates $ \norm{\delta^\alpha( a)}_\mA \leq C_a ( 2 \pi \sigma)^{\abs
  \alpha}$ and   $\norm{\delta^\alpha( b)}_\mA \leq C_b ( 2 \pi
\tau)^{\abs \alpha}$. The resulting   estimate is 
\begin{equation*}
\norm{\delta^\alpha(ab)}_\mA \leq \sum_{\beta \leq\alpha } \binom{
  \alpha}{\beta} C_a C_b (2 \pi \sigma)^{|\beta|} (2 \pi \tau)^{
  |\alpha-\beta| }=C_aC_b (2 \pi (\sigma+\tau))^{\abs \alpha}\, ,
\end{equation*}%
therefore  $ab$ is $\sigma + \tau$-\BL.
  \end{proof}
  Next we formulate a theorem of Jackson-Bernstein type for \BA
  s. 
  \begin{thm}\label{prop:jacksonbernstein}
    Let $\mA$ be a \BA\ with automorphism group $\Psi$, and $\{X_\sigma : \sigma \geq 0\}$ be the approximation
    scheme of bandlimited elements.  Then, for $r >0$,
    \begin{equation}
      \label{eq:cha1}
      \Lambda _r (\mA ) = \mE ^\infty _r (\mA ) \, .
    \end{equation}
    In other words, $a \in \Lambda_r(\mA)$, if and only if $ E_\sigma(a) \leq C \sigma^{-r}$ for all $ \sigma>0$.
  \end{thm}
  We will split the proof into several statements. 
  One of the main tools will be smooth approximating units in $\mA$, which we will review next.

  \label{prop:lonemod}
  Given $\mu \in \mM(\bbRd)$ and $a \in C(\mA)$, the \emph{action of} $\mu$ on $a$ is defined by
  \begin{equation}\label{eq:module}
    \mu * a = \int_{\bbRd}\psi_{-t}(a) d\mu(t). 
  \end{equation}
  This action is a generalization of the usual convolution and satisfies similar properties.  See
  {\cite{Butzer67}} for details and
  proofs.

If the group action is periodic with period one, 
the action of $\mu$ on $a$ is
\begin{equation}
  \label{eq:16}
  \mu * a = \int_{\bbTd} \psi_{-t}(a) \, d\mu(t)=  \sum_{k \in \bbZd} \mF(f)(k) \hat a(k),
\end{equation}
where  $\hat a (k)$ is the $k$-th Fourier coefficient of $a$ and the sum converges in  the C1-sense as in Proposition~\ref{prop_DeLeeuw}.\\
  (i) If $ a \in C(\mA), \mu \in \mM(\bbRd)$, then
  \begin{equation}
    \label{eq:4}
    \norm{\mu * a}_\mA \leq M_\Psi 
    \, \norm{\mu}_{M(\bbRd)}\, \norm{a}_\mA  \, .
  \end{equation} 
  (ii) If $f \in C_c^\infty(\bbR^d)$ and $a \in C(\mA)$, then
  \begin{equation}
    \label{eq:3}
    \delta^\alpha(f * a)=D^\alpha f *a \in C(\mA)
  \end{equation}
  for every multi-index $\alpha$.  In particular, $f * a \in C^\infty(\mA)$.
\\
  (iii) Taylor's formula: if $a \in C^k(\mA)$, then
  \begin{equation}
    \label{eq:taylor}
    \psi_t(a)=\sum_{\abs \alpha \leq k} \frac{\delta^\alpha(a)}{\alpha !} t^\alpha + \frac{R_k(t)}{k!}\abs{t}^k   ,
  \end{equation}
  where the remainder term $R_k(t)$ is bounded by the modulus of continuity
  \begin{equation}
    \label{eq:remainder}
    \abs{R_k(t)} \leq C \max_{\abs {\alpha}=k}\omega_{|t|}(\delta^\alpha a).
  \end{equation}
Taylor's formula for $\psi_t(a)$ and the estimation of the remainder follow from the Taylor expansion of $G_{a',a}(t)=\inprod{a',\psi_t(a)}$ (see, e.g.~\cite[3.2]{krantz83} for the scalar case).

  For the construction of approximating units let $f_\rho (x)= \rho^{-d} f(\rho\inv x)$, $\rho >0$, be the dilation
  of $f\in L^1(\bbRd)$. Then
  \begin{equation}
    \label{eq:apprunit}
    f_\rho * a = \int_{\bbR^n}{\psi_{-\rho u}(a) f(u) \, du}\,.
  \end{equation}

  \begin{lem}\label{lem:smoothappr}
    Let \mA\ be a \BA,  $a \in C(\mA)$, and  $\kappa \in L^1(\bbRd)$ with  $\int_{\bbRd}\kappa(x) \, dx=1$. 
    \begin{enumerate}
   \item [(i)] If $\kappa \in L^1_{v_1}(\bbRd)$, then
    \begin{equation}
      \label{eq:modcontapp}
      \norm{a-\kappa_\rho *a}_\mA \le C \omega_\rho(a).
    \end{equation}
    \item [(ii)] If $\kappa \in L^1_{v_{k+1}}(\bbRd)$, and if $\int_{\bbRd}{t^\alpha \kappa(t) \, dt}=0$ for $1 \le \abs {\alpha}
    \le  k$, $k \in \bbN$, then for every  $a \in C^{ k}(\mA)$
    \begin{equation}
      \label{eq:smoothapp}
      \norm{a-\kappa_\rho *a}_\mA \le C \rho^{ k } \max_{\abs \beta
        =k}\omega_\rho (\delta^\beta (a)) \, .
    \end{equation}
  \end{enumerate}
\end{lem}
  \begin{proof} 
    The proof is similar to standard approximation results for $C(\bbT)$ or $C_u(\bbR^d)$.
    Part (i) follows from
    \begin{equation*}
      \norm{a-\kappa_\rho *a}_\mA \leq\int_{\bbR^n}\abs{\kappa(u)} \, \norm{a -\psi_{-\rho u}(a)}_\mA  \, du \leq 
      \int_{\bbR^n}{\abs{\kappa(u)} \, \omega_{\rho |u|}(a) \, du},
    \end{equation*}
    and the property
    \begin{equation}\label{eq:38}
      \omega_{\rho |u|}(a) \leq \sup_{t\in \bbRd } \norm{\psi_t} \,
      (1+\abs u ) \omega_\rho(a) \, ,
    \end{equation}
    which is proved as in the scalar case~\cite[1.2.1]{Trigub04}.  The proof of (ii) uses Taylors formula
    (\ref{eq:taylor}) in connection with the vanishing moment condition 
and the estimation of the
    remainder (\ref{eq:remainder}), see~\cite[1.2.6]{Trigub04}, \cite[3.3]{krantz83} for details.
  \end{proof}
  We need another property of the spectrum.
  \begin{lem}\label{lem:convspec}
    If $a \in C(\mA)$ and $f \in L^1(\bbRd)$, then
    \begin{equation}
      \label{eq:31}
      \spec (f *a) \subseteq \supp(\mF f) \cap \spec(a). 
    \end{equation}
  \end{lem}
  \begin{proof}
    By definition $t \in \spec(f *a)$ means that there exists $a' \in \mA'$, such that $t \in \supp \mF G_{a', f
      * a}$.  An elementary calculation shows that
    \begin{equation}\label{eq:46}
      G_{a', f * a}= f*G_{a',  a}
    \end{equation}
    for all $a'\in\mA'$, $a \in C(\mA)$. So $\supp \mF G_{a', f * a} \subseteq \supp \mF f \cap \supp \mF G_{a',
      a}$, and the Lemma follows.
  \end{proof}

  With the existence of approximating kernels we can now state a Jackson-type theorem for automorphism groups.

  \begin{prop} \label{prop:jackson1} Let $a \in \mA$ and $\sigma>0$.
\begin{enumerate}
   \item [(i)] Then there is a $\sigma$-\BL\ element $a_\sigma \in C(\mA )$, such that
    \begin{equation*}
      \norm{a-a_\sigma}_\mA \leq C \omega_{1/\sigma}(a)
    \end{equation*}
    with $C$ independent of $\sigma$ and $a$.
   \item [(ii)] If $a \in C^k(\mA)$, then there exists a $\sigma$-\BL\ element $a_\sigma \in \mA$, such that
    \begin{equation*}
      \norm{a-a_\sigma} \leq C \sigma^{-\abs
        \alpha}\max_{\abs\alpha =k}\omega_{1/\sigma}(\delta^\alpha a)\, .
    \end{equation*}
  \end{enumerate}
\end{prop}
  \begin{proof}
    (i) We follow~\cite[4.4.3]{Trigub04}.  Let $\kappa \in S(\bbRd)$, $\int_{\bbRd} \kappa=1$, $\supp \, \mF \kappa \subseteq
    [-1,1]^d$.  By Lemma~\ref{lem:smoothappr}(i)
    \begin{equation*}
      \norm{a-\kappa_{1/\sigma}*a}_\mA \leq C \omega_{1/\sigma}(a).
    \end{equation*}
    Since $\supp \mF(\kappa_{1/\sigma}) \subseteq [-\sigma ,\sigma]^d$, Lemma \ref{lem:convspec} implies that
    $\kappa_{1/\sigma}*a$ is $\sigma$-\BL, and we are done.

    (ii) The proof is similar. We choose the kernel $\kappa\in
    S(\bbRd)$ such that $\int_{\bbRd}{t^\alpha \kappa(t) \, dt}=0$ for $1 \le \abs {\alpha}
    \le  k$,  and then use part (ii) of Lemma~\ref{lem:smoothappr} instead of part (i).
  \end{proof}

  We draw two consequences of Proposition~\ref{prop:jackson1}. The first one is a density result in the style of Weierstrass'
  theorem, the second one is a Jackson type theorem that proves one half of the fundamental theorem~\ref{prop:jacksonbernstein}.
  \begin{cor}[Weierstrass] \label{prop:weierstrassBL} The set of \BL\ elements is dense in $C(\mA)$. Since
    $C^k(\mA )$ contains the \BL\ elements, $C^k(\mA )$ is also dense in $C(\mA )$.
  \end{cor}

  \begin{cor} \label{p:ch47} If $a \in \Lambda_r(\mA)$ for $r >0$, then $E_\sigma(a)=\bigo(\sigma^{-r})$.
  \end{cor}

\begin{proof}
  For $0<r <1$ this follows immediately from Proposition~\ref{prop:jackson1}(i) and the definition of $\Lambda _r
  (\mA )$. For $r >0, r \not \in \bbZ$, we use Proposition~\ref{prop:jackson1}(ii).  The proof for $r \in \bbZ $
  is similar. We have to assume in addition that the approximation kernel $\kappa$ is even. See~\cite[3.5,
  3.6]{krantz83} for the necessary details.
\end{proof}

Before proving the converse implication in Theorem~\ref{prop:jacksonbernstein}, i.e., the Bernstein-type result,
we need a mean-value property of automorphism groups.
\begin{lem} \label{lem:ch45} If $a$ is $\sigma$-\BL, then
  \begin{equation}
    \label{eq:49}
    \norm{\Delta_ta}_\mA \leq C\sigma\,  \abs{t}\, \norm{a}_\mA \, .
  \end{equation}
\end{lem}
\begin{proof}
  We use a weak-type argument.
  \[
  \begin{split}
    \norm{\Delta_t a}_\mA&= \sup_{\norm{a'}\leq 1}\abs{\langle a', \psi _t(a) - a\rangle } = \sup_{\norm{a'}\leq
      1}\Bigl\lvert{\int_0^1 \nabla G_{a',a}(\lambda
      t)\cdot t \; d\lambda} \Bigr\rvert\\
    & \leq\sup_{\norm{a'}\leq 1} C \abs{t}_2\norm{\abs{\nabla G_{a',a}}_2}_\infty \, .
  \end{split}
  \]
  Since $G_{a',a}$ is bandlimited, Bernstein's inequality for scalar functions yields that $\|\, |\nabla
  G_{a',a}|_2 \|_{\infty } \leq C \sigma \|G_{a',a}\|_\infty $. We may continue the estimate by
\[
\begin{split}
\norm{\Delta_t a}_\mA & \leq C \abs{t}_2 \sup_{\norm{a'}\leq 1}
\norm{\abs{\nabla G_{a',a}}_2}_\infty  \leq C_0 
\abs{t}_2 \sigma \sup_{\norm{a'}\leq 1} \norm{ G_{a',a}}_\infty\\
& \leq C_1 \abs{t}_2\, \sigma \, \|a\|_{\mA } 
\leq  C_2 \sigma \,\abs{t} \, \norm{a}_{\mA } \,.
\end{split}
\]
\end{proof}
 \begin{prop}\label{p:ch46} 
  Let $a \in \mA$, and $r >0$. If $E_\sigma(a) \leq C\sigma^{-r}$ for all $\sigma >0$, then $a \in
  \Lambda_r(\mA)$.
\end{prop}
\begin{proof}
  We sketch a proof along the lines of \cite{Butzer71} and assume that $0<r < 1$ first. Choose $2^k$-\BL\
  elements $a_k$, such that $ \norm{a-a_k}_\mA \leq C 2^{-r k}$. By the triangle inequality we get
  $\norm{a_{k+1}-a_k} \leq 2 \, C 2^{-r k}$, and therefore the series
  \begin{equation*}
    a=a_0+\sum_{k=0}^\infty(a_{k+1}-a_k)
  \end{equation*}
  converges in the norm of $\mA $. Set $b_0=a_0$, $b_k=a_{k+1}-a_k, k>0$. Then $b_k$ is \BL\ with bandwidth $2
  \cdot 2^k$, and $\norm{b_k}_\mA\leq 2C 2^{-r k}$.  We need an estimate for the norm of
  $\Delta_ta$, and start with 
  \begin{equation}
    \label{eq:33}
    \norm{\Delta_t a}_\mA \leq \sum_{k=0}^M\norm{\Delta_tb_k}_\mA + \sum_{k=M+1}^\infty\norm{\Delta_tb_k}_\mA.
  \end{equation}
  Lemma~\ref{lem:ch45} implies that
  \[
  \norm{\Delta_t b_k}_\mA\leq C 2^k \abs{t}\, \norm{b_k}_\mA \leq  C'
  2^{k-r k  }\abs{t}
  \]
  for all $k\in \bbN $.
  For $k>M$ we control the norm of the second sum in (\ref{eq:33}) with the triangle inequality
  \[
  \norm{\Delta_t b_k}_\mA \leq \norm{\psi_t b_k}_\mA+\norm{b_k}_\mA \leq (M_\Psi+1)\norm{b_k}_\mA \leq \tilde C 2^{-k r}.
  \]
  Substituting back into (\ref{eq:33}) yields
  \begin{equation}
    \label{eq:34}
    \norm{\psi_ta -a}_\mA \leq C''\Bigl(\abs t \sum_{k=0}^M 2^{k-r
      k}+\sum_{k=M+1}^\infty 2^{-r k}\Bigr) \leq C''\bigl(\abs t
    2^{M(1-r)}+ 2^{-r M}\bigr). 
  \end{equation}
  If we choose $M$ such that $1\leq 2^M\abs{t} < 2$, then
  \begin{equation}
    \label{eq:36}
    \norm{\psi_ta -a}_\mA \leq C_1 2^{-r M} \leq C_1  \abs{t}^r,
  \end{equation}
  and $a \in \Lambda_r(\mA)$, as desired.

  Next consider the case $r = m+\eta $ for $m\in \bbN $ and $0<\eta
  <1$. By \eqref{xxx}  we have 
$$
\|\delta ^\alpha (b_k)\|_{\mA } \leq C (2\pi 2^{k+1})^{|\alpha |}
\|b_k\|_{\mA } \leq C' 2^{(k+1)(|\alpha | - r)}
$$
for all $k\in \bbN $ and $\alpha \in \bbN ^d_0$. Consequently the series
$\sum _{k=0} ^\infty \delta ^\alpha b_k$ converges in $\mA $ for all
$\alpha $ with $|\alpha | \leq m$ and its sum must be $\delta ^\alpha
(a)$ (because each $\delta _j$ is closed on $\mD (\delta ^\alpha )$). 
We now apply the above estimates~\eqref{eq:33} and \eqref{eq:34} with
$\delta ^\alpha (a)$ instead of $a$  and deduce that 
$\delta ^\alpha (a) $ must be in $\Lambda _\eta (\mA )$ for $|\alpha |
\leq k$. Thus $a\in \Lambda _r(\mA )$. 

  If $r$ is an integer, then we have to use second order differences and a corresponding version of the mean
  value theorem. The argument is almost the same as above
  (see~\cite{Trigub04} for details in the scalar case). 
\end{proof}

Combining Propositions~\ref{p:ch47} and~\ref{p:ch46}, we have completed the proof of our main theorem
(Theorem~\ref{prop:jacksonbernstein}).

\subsection{Littlewood-Paley Decomposition}
\label{sec:littl-paley-decomp}
We derive a \LP\ characterization of $\Lambda_r(\mA)$ to obtain explicit expressions for the norm of \HZ\ spaces
on \MA s.
  
Let $\psi \in S(\bbRd)$ with
$ \supp \psi \subseteq \set{\xi \in \bbRd \colon 2^{-1} \leq \abs{\xi}_2 \leq 2}$, $\psi(\xi) >0$ for $ 2^{-1} <
\abs{\xi}_2 <2$, and $\sum_{k \in \bbZ}\psi(2^{-k}\xi)=1$ for all $ \xi \in \bbRd\setminus \set{0}$. Set
$\phi_{k+1}(\xi)=\psi(2^{-k}\xi), k\geq 0$ and $\phi_0(\xi )
=1-\sum_{k<0}\psi (2^{-k}\xi )$.
Call $\set{\phi_k}_{k \geq 0}$ a dyadic partition of unity. 
\begin{prop} \label{prop:littl-paley-decomp-3} Let $a \in \mA$ and let $\set{\phi_k}$ be a dyadic partition of
  unity, $\Phi_k=\mF^{-1}\phi_k$, and $r>0$.  Then $a \in
  \Lambda_r(\mA)$,  if and only if
  \begin{equation}
    \label{eq:39}
    \sup_{k \geq 0} 2^{r k}\norm{\Phi_k * a}_\mA
  \end{equation}
  is finite, and (\ref{eq:39}) defines an equivalent norm on $\Lambda_r(\mA)$.
\end{prop}
\begin{proof}
  We combine the norm equivalence $\norm{a}_{\Lambda_r(\mA)} \asymp \sup_{\norm{a'}_{\mA'}\leq 1}
  \norm{G_{a',a}}_{\Lambda_r(\bbRd)}$ (Lemma \ref{lemma:weakdeff}) with the classical \LP\ characterization of $\Lambda_r(\bbRd)$ (see, e.g.,
  \cite{bergh76}): 
  \begin{equation}
    \label{eq:43}
    \norm{G_{a',a}}_{\Lambda_r(\bbRd)} \asymp \sup_{k\geq 0} 2^{k
      r}\norm{\Phi_k * G_{a',a}}_\infty \, .
  \end{equation}
  As $\Phi_k * G_{a',a}= G_{a',\Phi_k*a}$ by (\ref{eq:46}), we obtain
  \[
  \norm{a}_{\Lambda_r(\mA)} \asymp \sup_{\norm{a'}_{\mA'}\leq 1}\, \sup_{k\geq 0} 2^{k r}\abs{\inprod{a',\Phi_k
      *a}} \asymp \sup _{k\geq 0} 2^{kr } \|\Phi _k \ast a \|_{\mA } \, ,
  \]
  which is ~\eqref{eq:39}.
\end{proof}

\subsection{ A Characterization of \HZ\ Spaces in Matrix Algebras.}
\label{sec:hz-spaces-ma}
For  \MA s we may characterize $\Lambda_r(\mA)$ more explicitly with the help of
Proposition~\ref{prop:littl-paley-decomp-3}.
\begin{prop}
  Let \mA\ be a solid \MA. Then the norm on $\Lambda_r(\mA)$ is equivalent to the expression
  \[
  \max \Big( \norm{\hat A(0)}_\mA, \; \sup_{k\geq 0} 2^{kr} \norm{\sum_{2^k \leq \abs l < 2^{k+1}} \hat A(l)
  }_\mA \Big).
  \]
\end{prop}
\begin{proof}
  Let $\set{\phi_k}_{k\geq 0}$ be a \DPU\ and set $C_k =
  \norm{\sum_{2^{k}\leq |l| <2^{k+1}} \hat A(l)}_\mA $.
  Since $A\in \Lambda _r (\mA)$, the Fejer-means of the Fourier series $\chi _t(A) = \sum \hat A (l) e^{2\pi i l
    \cdot t} $ converge by Proposition~\ref{prop_DeLeeuw} and thus
  \begin{eqnarray*}
    \Phi _k \ast A &=& \int _{\bbRd} \Phi _k(t) \chi _{-t}(A) \, dt \\
    &=& \sum _{k\in \bbZd} \hat A (l)   \int _{\bbRd} \Phi _k(t)
    e^{-2\pi i l \cdot t} \, dt = \sum _{k\in \bbZd} \hat A (l) \,
    \phi _k (l) \, . 
  \end{eqnarray*}
  The solidity of \mA\  implies that, for $k\geq 0$,
  \[
  B_k = \norm{ \Phi _k \ast A} _{\mA } = \norm{\sum_{l \in \bbZd} \phi_k(l) \hat A(l)}_\mA \leq \norm{\sum_{2^{k-1} \leq \abs l < 2^{k+1}} \hat A(l)}_\mA =C_{k-1} + C_k
  \]
  On the other hand, since $\phi_{k-1}+\phi_k + \phi_{k+1} \equiv 1$ on $\set{\xi \colon 2^{k-1} \leq \abs \xi _2
    \leq 2^{k+1}}$, we obtain $C_k \leq B_{k-1}+B_k+B_{k+1}$. 
  Consequently $\norm{A}_{\Lambda _r (\mA )} = \sup _{k\geq 0} 2^{r k } B_k$ and the expression $\sup _{k\geq 0}
  2^{r k } C_k$ are equivalent norms on $\Lambda _r (\mA )$.
\end{proof}
For the standard matrix algebras the results can be detailed further. For the algebra of convolution-dominated
matrices defined in \eqref{eq:basknorm}, we obtain a completely new form of off-diagonal
decay. 
\begin{thm}\label{thm:appbask}
  Let $\mC _0$ denote the algebra of convolution-dominated matrices. The approximation space $\mE^\infty_r(\mC_0)$
  is equal to the H\"older-Zygmund space $\Lambda_r(\mC_0)$. It consists of all matrices $A$ satisfying
  \begin{equation}\label{eq:44}
    \norm{\hat A (0)}_{\ell^2 \to \ell^2}<\infty, \quad
    2^{rk}\sum_{2^k \leq \abs l < 2^{k+1}}\norm{\hat A
      (l)}_{\ell^2 \to \ell^2} = 2^{r k }\sum_{2^k \leq \abs l < 2^{k+1}}  \sup _{m\in
      \bbZd} |A(m,m-l)|  \leq C 
  \end{equation}
for all $k\geq 0$.
\end{thm}

We can now prove Theorem~\ref{thm:baskalgapp} of the introduction. Assume that the matrix $A$ satisfies (\ref{eq:44}) and is invertible on $\ell^2(\bbZd)$. By Theorem~\ref{thm:appbask}, $A \in \Lambda_r(\mC_0)$, and by Theorem~\ref{thm_HZ_is_ICBA} $\Lambda_r(\mC_0)$ is \IC\ in $\mC_0$ which in turn is \IC\ in \bopzd. 
Thus $A \inv \in \Lambda_r(\mC_0)$ and  $A \inv$ also satisfies the conditions~(\ref{eq:44}).

The characterization~\eqref{eq:44} implies the embeddings
\[
\mC_r \subset \Lambda_r(\mC_0) \subset \mC_s \quad \text{for } 0\leq s < r \, ,
\]
and it is easy to show that the three algebras are distinct.

The characterization of \HZ\ spaces in the Jaffard class is even simpler. The \LP\ characterization of
$\Lambda_r(\mJ _s)$ implies that
\begin{eqnarray*}
A \in \Lambda_r(\mJ _s) &\Leftrightarrow & \sup _{k\geq 0} 2^{rk}
\|\sum_{2^k \leq \abs l < 2^{k+1}}\hat A       (l) \|_{\mJ _r }
<\infty  \\
&\Leftrightarrow &    \norm{\hat A(l)}_{\mJ _s} \abs{l}^r \leq C' \\
& \Leftrightarrow &  \|\hat A (l)\|_{\ell^2 \to \ell^2}
(1+|l|)^{s+r } < C \text{ for all } l \Leftrightarrow A \in \mJ_{r+s} \, .
  \end{eqnarray*}
Thus we have another proof of the identification $ \Lambda_r(\mJ_s)=\mJ_{s+r}$ of
Proposition~\ref{prop:Jaffard-char}.
\\

\def\cprime{$'$} \def\cprime{$'$}


\begin{thebibliography}{10}

\bibitem{Almira01}
J.~M. Almira and U.~Luther.
\newblock Approximation algebras and applications.
\newblock In {\em Trends in approximation theory (Nashville, TN, 2000)}, Innov.
  Appl. Math., pages 1--10. Vanderbilt Univ. Press, Nashville, TN, 2001.

\bibitem{Almira06}
J.~M. Almira and U.~Luther.
\newblock Inverse closedness of approximation algebras.
\newblock {\em J. Math. Anal. Appl.}, 314(1):30--44, 2006.

\bibitem{Bas90}
A.~G. Baskakov.
\newblock Wiener's theorem and asymptotic estimates for elements of inverse
  matrices.
\newblock {\em Funktsional. Anal. i Prilozhen.}, 24(3):64--65, 1990.

\bibitem{Baskakov97}
A.~G. Baskakov.
\newblock Asymptotic estimates for elements of matrices of inverse operators,
  and harmonic analysis.
\newblock {\em Sibirsk. Mat. Zh.}, 38(1):14--28, i, 1997.

\bibitem{beals77}
R.~Beals.
\newblock Characterization of pseudodifferential operators and applications.
\newblock {\em Duke Math. J.}, 44(1):45--57, 1977.

\bibitem{bergh76}
J.~Bergh and J.~L{\"o}fstr{\"o}m.
\newblock {\em Interpolation spaces. {A}n introduction}.
\newblock Springer-Verlag, Berlin, 1976.
\newblock Grundlehren der Mathematischen Wissenschaften, No. 223.

\bibitem{Black98}
B.~Blackadar.
\newblock {\em {$K$}-theory for operator algebras}, volume~5 of {\em
  Mathematical Sciences Research Institute Publications}.
\newblock Cambridge University Press, Cambridge, second edition, 1998.

\bibitem{Brandenburg75}
L.~{B}randenburg.
\newblock {O}n identifying the maximal ideals in {B}anach algebras.
\newblock {\em {J}. {M}ath. {A}nal. {A}ppl.}, 50:489--510, 1975.

\bibitem{Bratteli72}
O.~Bratteli.
\newblock Inductive limits of finite dimensional {$C\sp{\ast} $}-algebras.
\newblock {\em Trans. Amer. Math. Soc.}, 171:195--234, 1972.

\bibitem{bratteli96}
O.~Bratteli.
\newblock {\em Derivations, dissipations and group actions on {$C\sp
  *$}-algebras}, volume 1229 of {\em Lecture Notes in Mathematics}.
\newblock Springer-Verlag, Berlin, 1986.

\bibitem{bratteli75}
O.~Bratteli and D.~W. Robinson.
\newblock Unbounded derivations of {$C\sp{\ast} $}-algebras.
\newblock {\em Comm. Math. Phys.}, 42:253--268, 1975.

\bibitem{bratteli76}
O.~Bratteli and D.~W. Robinson.
\newblock Unbounded derivations of {$C\sp{\ast} $}-algebras. {II}.
\newblock {\em Comm. Math. Phys.}, 46(1):11--30, 1976.

\bibitem{bratrob87}
O.~Bratteli and D.~W. Robinson.
\newblock {\em Operator algebras and quantum statistical mechanics. 1}.
\newblock Texts and Monographs in Physics. Springer-Verlag, New York, second
  edition, 1987.
\newblock $C\sp \ast$- and $W\sp \ast$-algebras, symmetry groups, decomposition
  of states.

\bibitem{Butzer67}
P.~L. Butzer and H.~Berens.
\newblock {\em Semi-groups of operators and approximation}.
\newblock Die Grundlehren der mathematischen Wissenschaften, Band 145.
  Springer-Verlag New York Inc., New York, 1967.

\bibitem{Butzer71}
P.~L. Butzer and R.~J. Nessel.
\newblock {\em Fourier analysis and approximation}.
\newblock Academic Press, New York, 1971.
\newblock Volume 1: One-dimensional theory, Pure and Applied Mathematics, Vol.
  40.

\bibitem{Butzer68}
P.~L. Butzer and K.~Scherer.
\newblock {\em Approximationsprozesse und {I}ntepolationsmethoden}.
\newblock B. I. Hochschulskripten 826/826a. Bibligraphisches Institut,
  Mannheim, 1968.

\bibitem{connes}
A.~Connes.
\newblock {\em Noncommutative geometry}.
\newblock Academic Press Inc., San Diego, CA, 1994.

\bibitem{deleeuw73}
K.~DeLeeuw.
\newblock Fourier series of operators and an extension of the {F}. and {M}.
  {R}iesz theorem.
\newblock {\em Bull. Amer. Math. Soc.}, 79:342--344, 1973.

\bibitem{Deleeuw75}
K.~De{L}eeuw.
\newblock {A}n harmonic analysis for operators. {I}: {F}ormal properties.
\newblock {\em {I}ll. {J}. {M}ath.}, 19:593--606, 1975.

\bibitem{Deleeuw77}
K.~{D}e{L}eeuw.
\newblock {A}n harmonic analysis for operators. {I}{I}: {O}perators on
  {H}ilbert space and analytic operators.
\newblock {\em {I}ll. {J}. {M}ath.}, 21:164--175, 1977.

\bibitem{Demko84}
S.~Demko, W.~F. Moss, and P.~W. Smith.
\newblock Decay rates for inverses of band matrices.
\newblock {\em Math. Comp.}, 43(168):491--499, 1984.

\bibitem{DeVore93}
R.~A. DeVore and G.~G. Lorentz.
\newblock {\em Constructive approximation}, volume 303 of {\em Grundlehren der
  Mathematischen Wissenschaften [Fundamental Principles of Mathematical
  Sciences]}.
\newblock Springer-Verlag, Berlin, 1993.

\bibitem{Engel00}
K.-J. Engel and R.~Nagel.
\newblock {\em One-parameter semigroups for linear evolution equations}, volume
  194 of {\em Graduate Texts in Mathematics}.
\newblock Springer-Verlag, New York, 2000.
\newblock With contributions by S. Brendle, M. Campiti, T. Hahn, G. Metafune,
  G. Nickel, D. Pallara, C. Perazzoli, A. Rhandi, S. Romanelli and R.
  Schnaubelt.

\bibitem{FGL06}
G.~Fendler, K.~Gr{\"o}chenig, and M.~Leinert.
\newblock Symmetry of weighted {$L\sp 1$}-algebras and the {GRS}-condition.
\newblock {\em Bull. London Math. Soc.}, 38(4):625--635, 2006.

\bibitem{Glimm60}
J.~G. Glimm.
\newblock On a certain class of operator algebras.
\newblock {\em Trans. Amer. Math. Soc.}, 95:318--340, 1960.

\bibitem{GKW89}
I.~Gohberg, M.~A. Kaashoek, and H.~J. Woerdeman.
\newblock The band method for positive and strictly contractive extension
  problems: an alternative version and new applications.
\newblock {\em Integral Equations Operator Theory}, 12(3):343--382, 1989.

I.~Gohberg, M.~A. Kaashoek, and H.~J. Woerdeman.
\newblock The band method for positive and strictly contractive extension
  problems: an alternative version and new applications.
\newblock {\em Integral Equations Operator Theory}, 12(3):343--382, 1989.

\bibitem{Gro09}
K.~Gr{\"o}chenig.
\newblock {\em Wiener's Lemma: Theme and Variations. An Introduction to
  Spectral Invariance}.
\newblock Applied and Numerical Harmonic Analysis. Birkh\"auser, Boston, 2009.
\newblock Inzell Lectures on Harmonic Analysis.

\bibitem{GL04a}
K.~Gr{\"o}chenig and M.~Leinert.
\newblock Symmetry and inverse-closedness of matrix algebras and functional
  calculus for infinite matrices.
\newblock {\em Trans. Amer. Math. Soc.}, 358(6):2695--2711 (electronic), 2006.

\bibitem{Gr08}
K.~Gr{\"o}chenig and Z.~Rzeszotnik.
\newblock Almost diagonalization of pseudodifferential operators.
\newblock {\em Ann. Inst. Fourier}, 58(6):2279--2314, 2008.

\bibitem{grushka07}
Y.~Grushka and S.~Torba.
\newblock Direct theorems in the theory of approximation of vectors in a
  {B}anach space with exponential type entire vectors.
\newblock {\em Methods Funct. Anal. Topology}, 13(3):267--278, 2007.

\bibitem{hille57}
E.~Hille and R.~S. Phillips.
\newblock {\em Functional analysis and semi-groups}.
\newblock American Mathematical Society Colloquium Publications, vol. 31.
  American Mathematical Society, Providence, R. I., 1957.
\newblock rev. ed.

\bibitem{Hul72}
A.~Hulanicki.
\newblock On the spectrum of convolution operators on groups with polynomial
  growth.
\newblock {\em Invent. Math.}, 17:135--142, 1972.

\bibitem{Jaffard90}
S.~Jaffard.
\newblock Propri\'et\'es des matrices ``bien localis\'ees'' pr\`es de leur
  diagonale et quelques applications.
\newblock {\em Ann. Inst. H. Poincar\'e Anal. Non Lin\'eaire}, 7(5):461--476,
  1990.

\bibitem{Katznelson}
Y.~Katznelson.
\newblock {\em An introduction to harmonic analysis}.
\newblock Cambridge Mathematical Library. Cambridge University Press,
  Cambridge, third edition, 2004.

\bibitem{MR1221047}
E.~Kissin and V.~S. Shul{\cprime}man.
\newblock Dense {$Q$}-subalgebras of {B}anach and {$C\sp *$}-algebras and
  unbounded derivations of {B}anach and {$C\sp *$}-algebras.
\newblock {\em Proc. Edinburgh Math. Soc. (2)}, 36(2):261--276, 1993.

\bibitem{kissin94}
E.~Kissin and V.~S. Shul{\cprime}man.
\newblock Differential properties of some dense subalgebras of {$C\sp
  \ast$}-algebras.
\newblock {\em Proc. Edinburgh Math. Soc. (2)}, 37(3):399--422, 1994.

\bibitem{Klo09}
A.~Klotz.
\newblock {\em Noncommutative Approximation: Smoothness and Approximation and
  Invertibility in Banach Algebras}.
\newblock PhD thesis, University of Vienna, 2009.

\bibitem{krantz83}
S.~G. Krantz.
\newblock Lipschitz spaces, smoothness of functions, and approximation theory.
\newblock {\em Exposition. Math.}, 1(3):193--260, 1983.

\bibitem{Kur90}
V.~G. Kurbatov.
\newblock Algebras of difference and integral operators.
\newblock {\em Funktsional. Anal. i Prilozhen.}, 24(2):87--88, 1990.

\bibitem{Nikolski99}
N.~Nikolski.
\newblock In search of the invisible spectrum.
\newblock {\em Ann. Inst. Fourier (Grenoble)}, 49(6):1925--1998, 1999.

\bibitem{Pietsch81}
A.~Pietsch.
\newblock Approximation spaces.
\newblock {\em J. Approx. Theory}, 32(2):115--134, 1981.

\bibitem{Rabinovich98}
V.~S. Rabinovich, S.~Roch, and B.~Silbermann.
\newblock Fredholm theory and finite section method for band-dominated
  operators.
\newblock {\em Integral Equations Operator Theory}, 30(4):452--495, 1998.
\newblock Dedicated to the memory of Mark Grigorievich Krein (1907--1989).

\bibitem{rrs04}
V.~S. Rabinovich, S.~Roch, and B.~Silbermann.
\newblock {\em Limit operators and their applications in operator theory},
  volume 150 of {\em Operator Theory: Advances and Applications}.
\newblock Birkh\"auser Verlag, Basel, 2004.

\bibitem{Sjo95}
J.~Sj{\"o}strand.
\newblock Wiener type algebras of pseudodifferential operators.
\newblock In {\em S\'eminaire sur les \'Equations aux D\'eriv\'ees Partielles,
  1994--1995}, pages Exp.\ No.\ IV, 21. \'Ecole Polytech., Palaiseau, 1995.

\bibitem{Sun07a}
Q.~Sun.
\newblock Wiener's lemma for infinite matrices.
\newblock {\em Trans. Amer. Math. Soc.}, 359(7):3099--3123 (electronic), 2007.

\bibitem{Timan63}
A.~F. Timan.
\newblock {\em Theory of approximation of functions of a real variable}.
\newblock Translated from the Russian by J. Berry. English translation edited
  and editorial preface by J. Cossar. International Series of Monographs in
  Pure and Applied Mathematics, Vol. 34. A Pergamon Press Book. The Macmillan
  Co., New York, 1963.

\bibitem{Torba08}
S.~Torba.
\newblock Inverse theorems in the theory of approximation of vectors in a
  {B}anach space with exponential type entire vectors.
\newblock arXiv:0809.4030v1 [math.FA], 2008.

\bibitem{Trigub04}
R.~M. Trigub and E.~S. Bellinsky.
\newblock {\em Fourier analysis and approximation of functions}.
\newblock Kluwer Academic Publishers, Dordrecht, 2004.

\end{thebibliography}

\end{document}